\documentclass[12pt, oneside, reqno]{amsart}
\usepackage[margin=2.3cm]{geometry}

\usepackage{amsfonts,color}
\usepackage{mathabx}
\usepackage{amssymb,amscd}
\usepackage{footnote}
\usepackage[utf8]{inputenc}
\usepackage{mathtools}
\usepackage{amsthm,wasysym}
\usepackage[english]{babel}
\usepackage{bbm}
\usepackage{colonequals}
\usepackage{cancel}
\usepackage{dsfont}
\usepackage{adjustbox}
\usepackage{eucal}
\usepackage{graphicx}
\usepackage{bm}
\usepackage{amsfonts}
\usepackage{hyperref}
\usepackage{wrapfig}
\usepackage{subcaption}
\usepackage{amsmath}
\usepackage{float}
\usepackage{stmaryrd}
\usepackage{accents}
\usepackage{pdfsync}
\usepackage{todonotes}

\usepackage{algorithm}
\usepackage{algpseudocode}

\usepackage{siunitx}
\newcolumntype{N}{c@{}S}

\usepackage{pgf,tikz,pgfplots}
\pgfplotsset{compat=1.15}
\usepackage{mathrsfs}
\usetikzlibrary{arrows}
\usetikzlibrary {arrows.meta} 
\usetikzlibrary{calc}

\usepackage{chngcntr}
\counterwithin{equation}{section}

\newtheorem{theorem}{Theorem}[section]
\newtheorem{lemma}[theorem]{Lemma}

\newtheorem{corollary}[theorem]{Corollary}
\newtheorem{proposition}[theorem]{Proposition}
\theoremstyle{definition}
\newtheorem{definition}[theorem]{Definition}
\theoremstyle{remark}
\newtheorem{remark}[theorem]{Remark}

\allowdisplaybreaks[1]

\newcommand{\da}{\mathrm{da}} 
\newcommand{\dl}{\mathrm{dl}} 

\newcommand{\vG}{\varGamma}
\newcommand{\VV}{\mathcal{V}}
\newcommand{\Vo}{\mathring{\VV}}
\newcommand{\trans}{\prime}
\newcommand{\om}{\varOmega}
\renewcommand{\Omega}{\om}
\def\d{\partial}
\newcommand{\veps}{\varepsilon}
\newcommand{\Xm}[1]{\mathfrak{X}({#1})}
\renewcommand{\forall}{\text{ for all }}

\newcommand{\Sc}{\ensuremath{\mathcal{S}}}
\newcommand{\E}{\ensuremath{\mathscr{E}}}
\newcommand{\Eint}{\ensuremath{\mathring{\mathscr{E}}}}
\newcommand{\Ebnd}{\ensuremath{\mathscr{E}_\partial}}
\newcommand{\V}{\ensuremath{\mathscr{V}}}
\newcommand{\Vint}{\ensuremath{\mathring{\mathscr{V}}}}
\newcommand{\Vbnd}{\ensuremath{\mathscr{V}_{\partial}}}

\newcommand{\dhat}[1]{{\bar{{#1}}}}
\newcommand{\gex}{{\dhat{g}}}  
\newcommand{\gpar}{{g}}  
\newcommand{\nabex}{\dhat{\nabla}}  
\newcommand{\Rmex}{\dhat{\Riemann}}
\newcommand{\Kex}{\dhat{K}}

\newcommand{\gappr}{{g_h}}

\renewcommand{\div}{\mathrm{div}}
\DeclareMathOperator{\curl}{\mathrm{curl}}
\DeclareMathOperator{\rot}{\mathrm{rot}}
\DeclareMathOperator{\inc}{\mathrm{inc}}
\newcommand{\W}{\Lambda}
\newcommand{\TT}{\mathcal{T}}

\newcommand{\mt}[1]{[{#1}]}

\newcommand{\jmp}[1]{\ensuremath{\llbracket #1 \rrbracket}}

\newcommand{\vol}[2][]{\omega_{{#2}}(\ifthenelse{\equal{#1}{}}{g}{#1})}
\newcommand{\vo}[1]{\omega_{{#1}}}
\newcommand{\volform}{\omega}
\newcommand{\volformh}{\omega_h}
\newcommand{\volformex}{\dhat{\omega}}
\newcommand{\voex}[1]{\volformex_{{#1}}}
\newcommand{\Kog}{\widetilde{\Gauss\omega}}

\newcommand{\dincop}[1][]{\widetilde{\inc_{{\ifthenelse{\equal{#1}{}}{g}{#1}}}}}
\newcommand{\drotrotop}[1][]{\widetilde{\rot\rot_{{\ifthenelse{\equal{#1}{}}{g}{#1}}}}}

\newcommand{\mat}[1]{\ensuremath{\begin{pmatrix}#1\end{pmatrix}}}	

\newcommand{\R}{\ensuremath{\mathbb{R}}}

\newcommand{\T}{\ensuremath{\mathscr{T}}}

\newcommand{\Ltwo}[1][]{\ensuremath{L^2\ifthenelse{\equal{#1}{}}{}{(#1)}}}
\newcommand{\Linf}[1][]{\ensuremath{L^{\infty}\ifthenelse{\equal{#1}{}}{}{(#1)}}}
\newcommand{\Lp}[1][]{\ensuremath{L^{p}\ifthenelse{\equal{#1}{}}{}{(#1)}}}
\newcommand{\Wsp}[1][]{\ensuremath{W^{s,p}\ifthenelse{\equal{#1}{}}{}{(#1)}}}
\newcommand{\Wsph}[1][]{\ensuremath{W_h^{s,p}\ifthenelse{\equal{#1}{}}{}{(#1)}}}
\newcommand{\Hone}[1][]{\ensuremath{H^1\ifthenelse{\equal{#1}{}}{}{(#1)}}}
\newcommand{\Honeh}[1][]{\ensuremath{H^1_h\ifthenelse{\equal{#1}{}}{}{(#1)}}}
\newcommand{\Honez}[1][]{\ensuremath{H^1_0\ifthenelse{\equal{#1}{}}{}{(#1)}}}
\newcommand{\Hmone}[1][]{\ensuremath{H^{-1}\ifthenelse{\equal{#1}{}}{}{(#1)}}}
\newcommand{\Hsh}[1][]{\ensuremath{H^s_h\ifthenelse{\equal{#1}{}}{}{(#1)}}}
\newcommand{\Htwo}[1][]{\ensuremath{H^2\ifthenelse{\equal{#1}{}}{}{(#1)}}}
\newcommand{\Htwoz}[1][]{\ensuremath{H^2_0\ifthenelse{\equal{#1}{}}{}{(#1)}}}
\newcommand{\Hmtwo}[1][]{\ensuremath{H^{-2}\ifthenelse{\equal{#1}{}}{}{(#1)}}}

\newcommand{\Riemann}{\mathcal{R}}

\newcommand{\Gauss}{K}
\newcommand{\Gaussh}{\Gauss_h}
\newcommand{\GeodCurv}{\kappa}

\newcommand{\nv}{\nu}
\newcommand{\tv}{{{\tau}}}       
\newcommand{\gn}{{{\hat{\nu}}}} 
\newcommand{\gt}{{{\hat{\tau}}}}

\newcommand{\Regge}[1][]{\ensuremath{\mathrm{Reg}\ifthenelse{\equal{#1}{}}{}{(#1)}}}

\newcommand{\idop}{\mathrm{id}}

\newcommand{\pder}[2]{\ensuremath{\partial_{#2}{#1}}}
\newcommand{\tr}[2][]{\ensuremath{\,\mathrm{tr}\ifthenelse{\equal{#1}{}}{}{_{#1}}(#2)}}

\newcommand{\Pol}{\ensuremath{\mathcal{P}}}
\newcommand{\RegInt}[1][]{\mathcal{I}_{h}^{\Regge\ifthenelse{\equal{#1}{}}{}{,#1}}}
\newcommand{\LtwoInt}[1][]{\Pi_{\Ltwo}^{\ifthenelse{\equal{#1}{}}{}{#1}}}
\newcommand{\LtwoProj}[1][]{P_{\Ltwo}^{\ifthenelse{\equal{#1}{}}{}{#1}}}

\newcommand{\Eucl}{\delta}

\newcommand{\nrm}[1]{|\!|\!|{{#1}}|\!|\!|}

\title[Improved convergence of distributional Gauss curvature]{On the improved convergence of lifted distributional Gauss curvature from Regge elements}

\author[J.~Gopalakrishnan]{Jay~Gopalakrishnan}
\address{Portland State University, PO Box 751, Portland OR 97201, USA}
\email{gjay@pdx.edu}

\author[M.~Neunteufel]{Michael~Neunteufel}
\address{Portland State University, PO Box 751, Portland OR 97201, USA}
\email{mneunteu@pdx.edu}

\author[J.~Sch\"oberl]{Joachim~Sch\"oberl}
\address{Institute of Analysis and Scientific Computing, TU Wien, Wiedner Hauptstr. 8-10, 1040 Wien, Austria}
\email{joachim.schoeberl@tuwien.ac.at}

\author[M.~Wardetzky]{Max~Wardetzky}
\address{Institute of Numerical and Applied Mathematics, University of G\"ottingen, Lotzestr. 16-18, 37083 G\"ottingen, Germany}
\email{email: wardetzky@math.uni-goettingen.de}

\begin{document}

\begin{abstract}
	Although Regge finite element functions are not continuous, useful
	generalizations of nonlinear derivatives like the
	curvature, can be defined using them. This paper is devoted to
	studying the convergence of the finite element lifting of a
	generalized (distributional) Gauss curvature defined using a metric
        tensor approximation in
	the Regge finite element space.  Specifically, we investigate the
	interplay between the polynomial degree of the curvature lifting by
	Lagrange elements and the degree of the metric
	tensor  in the Regge finite element space.
	Previously, a superconvergence result, where convergence
	rate of  one order higher than expected, was obtained when the
        approximate metric
	is the canonical Regge interpolant of the exact metric. In this
	work, we show that an even higher order can
	be obtained if the degree of the curvature lifting
	is reduced by one polynomial degree and if at least linear Regge elements are used. These improved convergence rates
	are confirmed by numerical examples.\\
  \vspace*{0.05cm}
	
\end{abstract}

\maketitle

\section{Introduction}
\label{sec:intro}

Substantial progress has recently been made on computing high-order
approximations of Gauss curvature on two-dimensional Riemannian manifolds using non-smooth metrics that are
piecewise smooth with respect to a mesh~\cite{GN2023, BKG21, GNSW2023}.
Perhaps the most well-known example is that of a piecewise constant
metric, where the angle defect at mesh vertices yields an approximation of Gauss curvature.  Being concentrated at vertices, this angle defect
can naturally be lifted into a linear Lagrange finite element space
that has one basis function per vertex. Taking this as a point of departure by viewing a piecewise constant metric as a polynomial of degree $k=0$ and the resulting Lagrange curvature lifting as a polynomial of degree $k+1 =1$ within each element,
we can generalize to higher degrees $k$.
Namely, in~\cite{GNSW2023}, we showed that if
a smooth metric is approximated using the canonical interpolant of the
Regge finite element space of degree $k$, then the error in a degree
$(k+1)$-Lagrange finite element approximation of the Gauss curvature
converges to zero in the $H^{-1}$-norm at the rate $\mathcal O (h^{k+1})$, where
$h$ is the mesh-size. The present work is devoted to answering the following related question. What convergence rates can be expected if we decide to approximate curvature in an
even higher degree Lagrange space---or for that matter, a lower degree
Lagrange space---while keeping the metric in the degree $k$ Regge
space? Since the analysis in~\cite{GNSW2023} used delicate
orthogonality properties (such as the orthogonality of the error in
Christoffel symbol approximations), the answer is not obvious.
In fact, the answer we provide in this paper may even
seem counterintuitive at first sight: reducing the degree of curvature
approximation to degree $k$ increases the convergence rate, while increasing
it to degree $k+2$ reduces the rate.
Specifically, we observe that,
under suitable assumptions, the following rates apply:
\begin{center}
  \begin{tabular}{|c|c|l|}
    \hline 
    Curvature approximation & $H^{-1}$ convergence & Source \\
    \hline
    $k+2$\qquad $k\ge 0$  & $\mathcal{O}(h^{k})$ & Section~\ref{sec:num_examples} in this paper
    \\
    $k+1$\qquad $k\ge 0$ & $\mathcal{O}(h^{k+1})$ & \cite[Theorem~6.5]{GNSW2023}
    \\
    $k\textcolor{white}{+0}$ \qquad $k\ge 1$ & $\mathcal{O}(h^{k+2})$ & Theorem~\ref{thm:improved_rate_Hm1}, \ref{thm:improved_rate_Hm1_pure} in this paper
    \\
    $k-1$\qquad $k\ge 2$ & $\mathcal{O}(h^{k+1})$ & Section~\ref{sec:num_examples} in this paper
    \\
    \hline 
  \end{tabular}
\end{center}
The remainder of this introduction places these and related prior
results into
perspective.

By Gauss' Theorema Egregium, Gauss curvature $\Gauss$ is an intrinsic quantity.
It can be computed considering solely the metric tensor of the manifold, without reference to any embedding. Therefore, it is natural to ask for discrete versions of Gauss curvature that arise when only approximations of the exact metric are given, and how well such discrete versions approximate the exact curvature when the approximated metrics are close to the exact one.

We consider Regge finite elements for discretizing the metric tensor. They originate from Regge calculus, originally developed for solving Einstein field equations in general relativity~\cite{Regge61}. Following Regge we consider a simplicial triangulation of the manifold and assign positive numbers to each edge. These numbers are interpreted as squared lengths and determine a piecewise constant metric tensor. Sorkin pointed out in \cite[Section II.A]{Sorkin75} that this piecewise constant metric tensor possesses tangential-tangential continuity, or $tt$-continuity  (which we define precisely in Subsection~\ref{ssec:regge-metric} below)  over element interfaces. Christiansen~\cite{Christiansen04} popularized 
Regge calculus in the study of  finite element methods (FEM) much the same way as 
{finite element exterior calculus} (FEEC) popularized the use of 
Whitney forms \cite{whitney57} in FEM.
He defined in \cite{Christiansen11} the lowest-order Regge finite element space (the $k=0$ case of the space
$\Regge_h^k$ defined in \eqref{eq:ReggeFE} below)
and showed that the linearization of the discretized Einstein--Hilbert action functional around the Euclidean metric equals the distributional incompatibility operator applied to such functions. Further, he proved in \cite{Christiansen13,Chr2024} that the densitized curvature of a sequence of mollified piecewise constant Regge metrics converges to the angle defect in the sense of measures. Li extended the Regge space  to arbitrary polynomial degrees $k$ and to higher-dimensional simplices \cite{li18}, and Neunteufel defined high-order Regge elements for quadrilaterals, hexahedra, and prisms \cite{Neun21}.

Due to the non-smoothness of the approximated metric $\gappr\in\Regge_h^k$ (which only has tangential-tangential continuity) and the nonlinearity of curvature, a  definition of consistent and convergent notion of discrete Gauss curvature is not obvious.
We refer to Sullivan \cite[Section 4.1]{Sullivan08} for a historical discussion and to Strichartz \cite[Corollary 3.1]{Str2020} for a definition of curvature as a measure on singular surfaces, where the curvature quantities, being multiplied by the corresponding volume forms,  are handled as densities. In  \cite{BKG21}, Berchenko-Kogan and Gawlik defined  a distributional version of  the densitized Gauss curvature, namely $\Gauss\volform$, a generalization of the product of Gauss curvature $K$ with the volume form~$\omega$
(and we analyzed their  curvature generalization further in~\cite{GNSW2023}).
In their work, in addition to the elementwise Gauss curvature $\Gauss|_T=\Gauss(\gappr)|_T$, they consider the jump of the geodesic curvature $\GeodCurv$ over edges and the angle defect at vertices as sources of Gauss curvature, i.e., for any $u$ in a space $\Vo(\T)$ based on a mesh $\T$ defined below, 
\begin{align}
	\label{eq:distr_Gauss}
\Kog(\gappr)(u) = \sum_{T\in \T}\int_T \Gauss|_T\,u\,\vo{T}+\sum_{E\in\E}\int_E\jmp{\GeodCurv}_E\,u\,\vo{E}+\sum_{V\in\V}\Theta_V\,u(V),
\end{align}
where $\vo{T}$ and $\vo{E}$ denote the volume forms of the respective (sub-)domains.
This allows for putting the well-established Gauss curvature approximation 
by angle deficit $\Theta_V$ ($2\pi$ minus the sum of the interior angles of triangles attached to the vertex) into a finite element context and to extend it to higher polynomial order. In fact, considering piecewise constant metrics, $\gappr\in\Regge_h^0$, the angle deficit is recovered, since then $\Kog(\gappr)(u) =\sum_{V\in\V}\Theta_V\,u(V)$.
The distributional Gauss curvature \eqref{eq:distr_Gauss} acts on piecewise smooth and globally continuous 0-forms defined by 
\begin{equation}	\label{eq:def_scalar_testfunc}
  \begin{gathered}
    \VV(\T) = \{ u\in \W^0(\T)\,:\, u \text{ is continuous}\},
    \\
    \Vo_{\Gamma}(\T) = \{ u\in \VV(\T)\,:\,\, u|_{\Gamma}=0\},\quad \Vo(\T)=\Vo_{\d\Omega}(\T).    
  \end{gathered}
\end{equation}
The meaning  of  ``piecewise smooth'' with respect to a ``mesh'' $\T$ and  
definition of  piecewise smooth $k$-form fields $\W^k(\T)$ appear in Subsection~\ref{ssec:regge-metric} below.
The standard degree~$k$ Lagrange finite element subspaces of
of the spaces in~\eqref{eq:def_scalar_testfunc} are denoted by
$\VV_h^k$, $\VV_{h,\Gamma}^k$, and $\Vo_h^k$, respectively,
Berchenko-Kogan and Gawlik proved~\cite{BKG21} error estimates in the $\Hmone$-norm by using an integral representation of \eqref{eq:distr_Gauss}. Indeed, let $\Eucl$ denote the Euclidean metric, whose coordinate components coincide with the classical Kronecker delta, $\mt{\delta}_{ij} = \delta_{ij}$ (not to be confused with the Dirac delta, which is never used in this work). Then there holds
\begin{align}
	\label{eq:int_rep}
	\Kog(g)(u) = \frac{1}{2}\int_0^1b(\Eucl+t(g-\Eucl);g-\Eucl,u)\,dt,
\end{align}
where the bilinear form $b(g;\sigma,u)$ is the covariant version of the Hellan--Herrmann--Johnson (HHJ) method \cite{Hel67,Her67,Joh73} extending the covariant differential operator $\div_g\div_g(\mathbb{S}_g\sigma)$ in the sense of distributions, where $\mathbb{S}_g\sigma=\sigma - \tr[g]{\sigma}g$. Recently, Gawlik and Neunteufel extended the analysis to the $\Hmtwo$-norm for the Gauss curvature, see \cite{GN2023}. Further, they considered an integral representation for the error $\big(\Kog(\gappr)-\widetilde{\Gauss\volform}(g)\big)(u)$.

It is often useful (or even necessary) to consider Gauss curvature as a function instead of as a functional or a distribution.
In \cite{Gaw20}, Gawlik computed a discrete Riesz representative $\Gaussh\in\Vo_h^r$ in the Lagrange finite element space $\Vo_h^r$ as a lifting of the distributional Gauss curvature \eqref{eq:distr_Gauss} via
\begin{align*}
	\int_{\Omega}\Gaussh\,u_h\,\volformh = \Kog(\gappr)(u_h) \qquad \forall u_h\in \Vo_h^r,
\end{align*}
where $\volformh=\sqrt{\det \gappr}\,dx^1\wedge dx^2$ denotes the approximated volume form.

He proved error estimates of this lifting also for Sobolev norms under the assumption that $\gappr\in\Regge_h^k$ is an optimal-order approximation of the exact metric $\gex$, with exact Gauss curvature $\Kex=\Gauss(\gex)$,
\begin{align}
	\label{eq:conv_rate_Gawlik}
	\|\Gaussh-\Kex\|_{H^l_h}\leq C\,h^{-l+\min\{k-1,r+1\}}\big(|\gex|_{H^{k+1}}+|\Kex|_{H^{r+1}}\big),\quad -1\leq l\leq r.
\end{align}
Here, $\|\cdot\|_{H_h^l}$ denotes the elementwise $H^l$-norm. In \cite{GNSW2023}, we considered an alternative integral representation that relies on the distributional covariant \emph{incompatibility} operator, $\inc_g=\curl_g\curl_g$, which is related to the HHJ method by $\inc_g\sigma=-\div_g\div_g(\mathbb{S}_g\sigma)$. We showed an  convergence rate increase (by one order) compared to \eqref{eq:conv_rate_Gawlik} if $\gappr$ is the canonical Regge interpolant (defined by Li in \cite{li18}, reproduced in \eqref{eq:RegInt} below) of the exact metric $\gex$ and $\Gaussh$ is assumed to be in $\Vo_h^{k+1}$,
\begin{align*}
	\|\Gaussh-\Kex\|_{H^l_h}\leq C\,h^{-l+k}\big(|\gex|_{W^{k+1,\infty}}+|\Kex|_{H^{k}}\big),\qquad -1\leq l\leq k.
\end{align*}
However, this convergence rate, when compared against  the best approximation capabilities  of the space, is not the theoretical optimum; for $\Gaussh\in \Vo_h^{k+1}$ we obtain only $\Ltwo$-convergence of order $k$ instead of $k+2$.

In this paper we show that an increased, optimal convergence rate for the lifting of the distributional Gauss curvature $\Gaussh$ and its densitized version $\Gaussh\volformh$ is obtained when considering Lagrange elements $\Vo_h^k$ of one polynomial degree less assuming at least linear elements, $k\geq 1$, are used. Our  analysis relies heavily on the properties of the canonical Regge interpolant \cite{li18} preserving specific moments at edges and elements. Therefore, the results only hold for the canonical Regge interpolant.
For more general metric approximations in  $\Regge_h^k$, 
the estimate  \eqref{eq:conv_rate_Gawlik} cannot generally be improved.
The technique of analysis in this work differs from our earlier work~\cite{GNSW2023}  in that we use an integral representation directly for the difference between the curvatures of the exact and the approximated metrics,  instead of employing an interpolation from the Euclidean metric as in \eqref{eq:int_rep}. This allows us to bypass a delicate ``Christoffel orthogonality property,'' which was a key step in our prior work (see~\cite[Lemma~6.10]{GNSW2023}). As in~\cite{GNSW2023}, our current analysis
  also  relies on the distributional covariant incompatibility operator
  $\dincop$, but now we rely specifically on its distributional $\Ltwo$-like adjoint $\drotrotop$. The latter  simplifies the curvature error analysis compared to \cite{GNSW2023} (even if it does not  provide estimates for  $\dincop$-approximation, which we did in \cite{GNSW2023}).

This paper is structured as follows. In the next section we quickly review   differential geometry notions we use, distributional covariant derivatives, and the distributional Gauss curvature. Section~\ref{sec:error_analysis} is devoted to the error analysis of the lifted (densitized) Gauss curvature in the $H^{-1}$- and stronger Sobolev norms.  In Section~\ref{sec:num_examples} we present numerical examples confirming the proved convergence rates.

\section{Notation}
\label{sec:notation}

Let $\Omega\subset\R^2$ be an open domain with a smooth metric tensor $\gex$ providing a Riemannian manifold structure $(\Omega,\gex)$. Consider a triangulation $\T$ of $\Omega$ consisting of possibly curved triangles. Denote the set of all edges and vertices by $\E$ and $\V$, respectively. We split $\E$ into edges lying on the boundary $\partial\Omega$, given by $\Ebnd$, and inner ones $\Eint=\E\setminus\Ebnd$. Analogously we define $\Vbnd$ and $\Vint$. We assume to be given an approximation of $\gex$, denoted $\gappr$, defined on the triangulation $\T$. The subscript $h$ indicates that $\gappr$ is defined with respect to the triangulation $\T$, where $h$ can be related to the maximal element size. All quantities computed from the exact metric $\gex$ will be marked by an overline ``\;$\dhat{\cdot}$\;'' throughout the paper.

\subsection{Regge metric}
\label{ssec:regge-metric}
Let $\Xm T$, $\W^k (T)$, and $\TT_l^k (T)$ denote the set
of smooth vector fields, $k$-form fields, and $(k, l)$-tensor fields
on a submanifold $T$ of $\om$, respectively. Here, smoothness signifies infinite
differentiability at interior points and continuous differentiability up to (including) the boundary.
In such symbols, replacement of the manifold $T$ by a collection of
subdomains such as the triangulation $\T$, yields the piecewise
smooth analogue with respect to the collection. For example, $\TT_l^k(\T)$
is the Cartesian product of $ \TT_l^k(T) $ over an enumeration of
all $T \in \T$. Analogously, $\W^1 (\T) = \TT_0^1(\T)$ and
$\Xm \T = \TT_1^0(\T)$. Let
$ \Sc(\T) = \{\sigma \in \TT_0^2(\T): \sigma(X, Y) = \sigma(Y, X)$ for
$X, Y \in \Xm \T \}$. Functions in $\Sc(\T)$ are symmetric covariant
2-tensors on $\om$ with no continuity over element interfaces in general. 
We define $\Sc^+(\T)$ as the subset of positive definite symmetric 2-tensors.
For coordinate computations, we use coordinates $x^1, x^2$ and Einstein's summation convention of repeated indices. Let the accompanying
coordinate frame and coframe be denoted by $\d_i$ and $dx^i$.
We assume that these coordinates preserve orientation, i.e., the orientation of $\om$ is given by the ordering $(\d_1, \d_2)$. We use standard operations on 2-manifold spaces such as the exterior derivative $d: \W^k(\om) \to \W^{k+1}(\om)$ (see e.g. \cite{Lee97, Peter16, Tu17}).

Every $E \in \Eint$ is of the form $E = \d T_+ \cap \d T_-$ for two elements
$T_\pm \in \T$.  We say that a $\sigma \in \Sc(\T)$ has
``tangential-tangential continuity'' or {\em ``$tt$-continuity''} if
$ \sigma|_{T_+}(X, Y) = \sigma|_{T_-}(X, Y)$ for all tangential vector fields $X, Y \in \Xm E$
for every $E$ in $\Eint$ (i.e., $\sigma(X, Y)$ is single-valued on all
$E \in \Eint$). This leads  to the definition of the (infinite-dimensional) {\em Regge space}
\begin{align}  \label{eq:ttspace}
	\Regge(\T)
	& =
	\{
	\sigma \in \Sc(\T): \sigma \text{ is $tt$-continuous}
	\}
\end{align}
and its  subset of \emph{Regge metrics}
\begin{align*}
	\Regge^+(\T) = \{ \sigma \in \Regge(\T): \sigma(X, X)>0 \forall 0\neq X \in \Xm \T\}.
\end{align*}
The approximate metric $\gappr$ is assumed to be in
$\Regge^+(\T)$. 

\subsection{Differential geometry}
\label{subsec:diffgeo}

Let the  unique Levi--Civita connection generated by $\gex$ be denoted by~$\nabex$. Note that it is standard to extend the Levi--Civita connection $\nabex$ from vector fields to tensor fields (see e.g., \cite[Lemma~4.6]{Lee97}) so that the Leibniz rule holds.

Following the sign convention of \cite{Lee97}, recall that the {\em Riemann curvature tensor} $\Rmex \in \TT^4_0(\om)$ of the manifold is defined by
\begin{align*}
	\Rmex(X, Y, Z, W) =
	\gex (\nabex_X \nabex_Y Z - \nabex_Y \nabex_X Z - \nabex_{[X,Y]}Z, W), \quad X, Y, Z, W \in \Xm \om.
\end{align*}
Recall that the {\em Gauss curvature} of $\om$ is given by 
\begin{align*}
	\Kex:=\Gauss(\gex) = 
	\frac{\Rmex(X, Y, Y, X)}{ \gex(X, X) \gex(Y, Y) - \gex(X, Y)^2},  
\end{align*}
where  $X$ and $Y$ are some linearly independent vector fields and the value $\Kex$ is independent of their choice.

We will also require the {\em geodesic curvature} along a curve $\varGamma$ in the
manifold $(\om, \gex)$. Let $\gt$ denote the $\gex$-normalized tangent vector of $\varGamma$ and $\gn$ the
$\gex$-orthonormal vector such that $(\gt,\gn)$ builds a right-handed coordinate system.  Then
\begin{align*}
	\dhat{\GeodCurv}:=\GeodCurv(\gex) = \gex ( \nabex_{\gt} \gt, \gn) = - \gex(\nabex_{\gt}\gn,\gt)
\end{align*}
gives the signed geodesic curvature of $\varGamma$. The element volume 2-form $\volformex$ and edge volume 1-form $\voex{E}$, $E\in\E$, read in coordinates
\begin{align}
	\label{eq:coo_volforms}
	\volformex = \sqrt{\det \gex}\,dx^1\wedge dx^2,\qquad \voex{E}=\sqrt{\gex(\tv,\tv)}\,d\tv,
\end{align}
where $\tv\in\Xm{E}$ denotes the Euclidean normalized tangent vector at edge $E$ and $d\tv$ is the associated 1-form. We use also the abbreviation of e.g. $\gex_{\tv\tv}:= \gex(\tv,\tv)$.

\subsection{Finite element spaces}
Let $\hat{T}\subset\R^2$ denote the reference triangle and define $\Pol^k(\hat{T})$ as the set of polynomials of degree up to $k$ on $\hat{T}$. For $T\in\T$ let $\Phi_T:\hat{T}\to T\in\Pol^k(\hat{T},\R^2)$ denote the diffeomorphic mapping from the reference to the physical element.

We define the Regge finite element space as a subspace of $\Regge(\T)$ \eqref{eq:ttspace} by
\begin{equation}
	\label{eq:ReggeFE}  
	\begin{aligned}
		\Regge_h^k =  \{ \sigma \in \Regge(\T):
		&
		\text{ for all }  T \in \T,\;
		\sigma|_T = \sigma_{ij} dx^i \otimes dx^j \text{ with } \sigma_{ij}\circ\Phi_T \in \Pol^k(\hat{T} ) \}.
	\end{aligned}
\end{equation}
Further, the Lagrange finite element space as a subspace of $\VV(\T)$ \eqref{eq:def_scalar_testfunc} is given by
\begin{align*}
	\begin{gathered}
		\VV_h^k = \{ u \in \VV(\T):  \text{ for all } T \in \T,
		\; u|_T\circ\Phi_T  \in \Pol^k(\hat T) \},
                \\
                \Vo_{h, \vG}^k = \{ u \in \VV_h^k: u|_{\vG} = 0 \}, \quad\text{ and } \quad
		\Vo_h^k = \Vo_{h, \d \Omega}^k.
	\end{gathered}
\end{align*}

\subsection{Lifted distributional Gauss curvature}
For the reader's convenience we derive the (lifted) distributional Gauss curvature following \cite{BKG21,GNSW2023}. Since a $\gappr\in\Regge^+(\T)$ is smooth within each element $T\in\T$ we can compute elementwise its Gauss curvature $\Gauss(\gappr)|_T$. It is only one contributor of the total distributional Gauss curvature as the jumps of $\gappr$ generate additional sources of curvature. Let for an edge $E\in\Eint$ the unique $\gappr$-normal vector that points inward to elements $T_{\pm}\in\T$, such that $E=\d T_+\cap \d T_-$, be denoted by $\gn^{T_{\pm}}_E$. As $\gappr$ is only $tt$-continuous, $\gn^{T_{+}}_E\neq -\gn^{T_{-}}_E$ in general. Thus, the jump of the geodesic curvature 
\begin{align*}
\jmp{\GeodCurv}_E= \GeodCurv_{\gn^{T_{+}}_E}+\GeodCurv_{\gn^{T_{-}}_E}
\end{align*} 
acts as a source of curvature at edges. If there is no chance of confusion we neglect the subscript and only write $\jmp{\cdot}$ for the jump over edges. 

Let $V\in\Vint$ be an interior vertex and $T\in\T$ a triangle containing $V$. Then there are two edges $E_{\pm}\in\Eint\cap\d T$ such that $V=\d E_+\cap \d E_-$. Denote $\gt_V^{E_\pm}$ the $\gappr$-normalized tangent vectors starting at $V$ and pointing into $E_{\pm}$. We define the following angle function on $V$
\begin{align*}
	\sphericalangle^T_V=\arccos(\gappr|_T(\gt_V^{E_+},\gt_V^{E_-}))
\end{align*}
and the angle deficit at vertex $V\in\Vint$
\begin{align}
	\label{eq:angle_deficit}
	\Theta_V=2\pi - \sum_{T\in \T_V}\sphericalangle^T_V,\qquad \T_V=\{T\in\T\,:\, V\in T\}.
\end{align}
This function acts as a source of curvature on vertices. Note that for the smooth metric $\gex$ there holds $\Theta_V=0$.

\begin{definition}
	\label{def:distr_Gauss}
	Let $\gappr\in\Regge^+(\T)$ be a Regge metric. The \emph{distributional densitized Gauss curvature} $\Kog(\gappr):\Vo(\T)\to\R$ is defined for all $u\in\Vo(\T)$
	\begin{align}
		\label{eq:def_distr_Gauss}
		\Kog(\gappr)(u) = \sum_{T\in \T}\int_T \Gauss|_T\,u\,\vo{T}+\sum_{E\in\Eint}\int_E\jmp{\GeodCurv}\,u\,\vo{E}+\sum_{V\in\Vint}\Theta_V\,u(V),
	\end{align}
where $\Gauss|_T$, $\GeodCurv$, $\vo{T}$, $\vo{E}$, and $\Theta_V$ are evaluated with respect to $\gappr$.
\end{definition}

\begin{remark}
  Note, that this generalizes  the \emph{densitized} Gauss curvature  $\Gauss\volform$  (see e.g., \cite{Christiansen13}), not solely  $\Gauss$. One can interpret \eqref{eq:def_distr_Gauss} as a measure with support on triangles, edges, and vertices, cf. \cite{Str2020}. See Remark~\ref{rem:pureK} for more on  approximating just~$K$.
\end{remark}

We consider a discrete Riesz representative of the functional in~\eqref{eq:def_distr_Gauss}, following Gawlik \cite{Gaw20}. We also incorporate essential and natural Dirichlet $\Gamma_D$ and Neumann $\Gamma_N$ boundary conditions, as discussed in \cite{GNSW2023}. To this end, extend the definition of the angle deficit \eqref{eq:angle_deficit} and the jump of the geodesic curvature to boundary vertices and edges in the obvious manner: we define $\jmp{\GeodCurv}_E=\GeodCurv$ for boundary edges $E\in\Ebnd$ and $\Theta_V$ for $V\in\Vbnd$ as in \eqref{eq:angle_deficit}. Note that $\jmp{\GeodCurv}_E$ and $\Theta_V$ do not vanish for smooth metrics $g$ at $E\in\Ebnd$ and $V\in\Vbnd$ in general. Instead, they are used to incorporate natural Neumann boundary conditions. We assume that on the Dirichlet boundary, the Gauss curvature $\Kex^D=\Kex|_{\Gamma_D}$ is prescribed. The Neumann boundary data is given by the functional $\widetilde{\kappa^N}:\VV(\T)\to\R$
\begin{align*}
	\widetilde{\kappa^N}(u) = \int_{\Gamma_N}\dhat{\GeodCurv}\,u\,\volformex_{\Gamma_N} + \sum_{V\in\V\cap \Gamma_N}\dhat{\sphericalangle}_V^N\, u(V),
\end{align*}
where $\dhat{\sphericalangle}_V^N$ denotes the exterior angle, which is $2\pi$ minus the interior angle, measured with respect to $\gex$ by the edges of $\Gamma_N$ at $V$.

\begin{definition}
	\label{def:lifted_distr_Gauss}
	Let $\gappr \in \Regge^+(\T)$, $k \ge 1$ be an integer, and assume that the Dirichlet data $\Kex^D$ is the trace of a Lagrange finite element function in $\VV_h^k$. The finite element curvature approximation $\Gaussh:=\Gaussh(\gappr)$ of degree $k$ is the unique function in $\VV_{h}^{k}$ determined by requiring that $\Gaussh|_{\Gamma_D} = \Kex^D$ on $\Gamma_D$ and for all $u_h\in \Vo_{h,\Gamma_D}^{k}$,
\begin{align} 
	\label{eq:lifted_distr_Gauss} 
		\int_{\Omega}\Gaussh \,u_h\,\volformh = \Kog(\gappr)(u_h)
		-\widetilde{\kappa^N}(u_h),
\end{align}
where we denote the volume form of $\gappr$ by $\volformh:=\vol[\gappr]{}$.
\end{definition}

Note that the  difference between Definition~\ref{def:lifted_distr_Gauss} and \cite[Definition 3.1]{GNSW2023} is the  decrease of degree of 
approximation space of $\Gaussh$ from $\VV_h^{k+1}$ to $\VV_h^k$ and the additional requirement of $k\geq 1$.

\subsection{Distributional covariant differential operators}
\label{subsec:distr_cov_ops}

In this section we review the definition of distributional covariant differential operators based on Regge metrics $g\in\Regge^+(\T)$. We focus on the incompatibility operator and its adjoint with their coordinate expressions. For an introduction and discussion we refer to e.g. to \cite[Section 4]{GNSW2023}. First, we focus on pointwise covariant differential operators for a given smooth metric $g\in\Sc^+(\Omega)$.

For a 1-form $\alpha\in\W^1(\Omega)$ and a (2,0)-tensor $\sigma\in\TT^2_0(\Omega)$ the covariant curl operators $\curl_g:\W^1(\Omega)\to\W^0(\Omega)$ and $\curl_g:\TT^2_0(\Omega)\to\W^1(\Omega)$ read in coordinates \cite{GNSW2023}
\begin{align*}
	\curl_g (\alpha)
	& =
	 \hat{\veps}^{ij} \pder{\alpha_j}{i},\\
	 \curl_g (\sigma)
	 & =
	 \hat{\veps}^{jk} (\pder{\sigma_{ik}}{j} - \Gamma_{ji}^m \sigma_{mk}) dx^i,
\end{align*}
where $\hat{\veps}^{ij}=\frac{1}{\sqrt{\det g}}\veps^{ij}$ and $\veps^{ij}$ denotes the permuting symbol being 1, -1, or 0 if $(i,j)$ is an even, odd, or no permutation of $(1,2)$, respectively. The covariant incompatibility operator $\inc_g=\curl_g\curl_g:\TT^2_0(\Omega)\to\W^0(\Omega)$ reads in coordinates
\begin{align*}
	\inc_g(\sigma)
	& =   \hat{\veps}^{qi} \hat{\veps}^{jk} \left(\d_j\d_q \sigma_{ik}
	-  \pder{( \Gamma_{ji}^m \sigma_{mk})}{q}
	- \Gamma_{lq}^l
	(\pder{\sigma_{ik}}{j} - \Gamma_{ji}^m \sigma_{mk})
	\right).
\end{align*}
Next, we consider for $f\in\W^0(\Omega)$ and $X\in \Xm{\Omega}$ the adjoint operators $\rot_g:\W^0(\Omega)\to\Xm{\Omega}$, $\rot_g:\Xm{\Omega}\to \TT^0_2(\Omega)$ and $\rot\rot_g=\rot_g\rot_g:\W^0(\Omega)\to\TT^0_2(\Omega)$. They read in coordinates \cite{GNSW2023}
\begin{subequations}
	\label{eq:rot_rotrot}
	\begin{align}
		\rot_g f & = \hat{\veps}^{iq} \d_q f\, \d_i =
		\frac{[\rot f]^i }{ \sqrt{\det g}} \d_i,
		\\ \label{eq:rotg-X}
		\rot_g X & = \hat{\veps}^{jq} (\d_q X^i + \Gamma^i_{qk} X^k)
		\d_i \otimes \d_j
		= \frac{[\rot [X]]^{ij} + \veps^{jq} \Gamma^i_{qk} X^k}{\sqrt{\det g}}
		\d_i \otimes \d_j,\\
	\label{eq:cov_rotrot_coo}
		\begin{split}
			\rot\rot_g f &= \rot_g(\rot_g f)^{ij}\d_i\otimes\d_j= \frac{[\rot [\rot_g f]]^{ij} + \veps^{jq} \Gamma^i_{qk} \mt{\rot_g f}^k}{\sqrt{\det g}}\d_i\otimes\d_j\\
		&=\frac{[\rot\rot f]^{ij} -[\rot f]^i\veps^{jq}\Gamma^l_{lq} + \veps^{jq} \Gamma^i_{qk} [\rot f]^k }{\det g}\d_i\otimes\d_j.
		\end{split}
\end{align}
\end{subequations}
In \eqref{eq:rot_rotrot} we used so-called vector and matrix proxies $\mt{\sigma}\in\R^{2\times2}$ and $\mt{X}\in\R^2$ for $\sigma\in \TT^0_2(\Omega)$ and $X\in\Xm{\Omega}$ \cite{Arnol18}. These proxies consist of coefficients in the coordinate basis expansions. For example, $\mt{\sigma}$ is the matrix, which $(i,j)$th entry is $\sigma^{ij}=\sigma(dx^i,dx^j)$. Then the standard two-dimensional Euclidean rotation operator applied to the vector $\mt{X}$ is $\rot\mt{X}^{ij}=\veps^{jk}\d_kX^i$.

There holds the integration by parts formulas for $f\in\W^0(\Omega)$, $\alpha\in \W^1(\Omega)$, $\sigma\in\TT^2_0(\Omega)$, and $X\in\Xm{\Omega}$ 
\begin{align*}
	\int_{\Omega}\langle\curl_g\sigma,X\rangle\,\volform &= \int_{\Omega}\langle\sigma,\rot_g X\rangle\,\volform + \int_{\d \Omega}\sigma(X,\gt)\,\vo{\d \Omega},\\
	\int_{\Omega}\curl_g\alpha\,f\,\volform &= \int_{\Omega}\langle\alpha,\rot_gf\rangle\,\volform + \int_{\d \Omega}\langle\alpha,\gt\rangle\,f\,\vo{\d \Omega},\\
	\int_{\Omega}\inc_g\sigma\,f\,\volform &= \int_{\Omega}\langle\curl_g\sigma,\rot_g f\rangle\,\volform + \int_{\d \Omega}\langle\curl_g\sigma,\gt\rangle\,f\,\vo{\d \Omega}\\
	&= \int_{\Omega}\langle\sigma,\rot\rot_g f\rangle\,\volform+\int_{\d \Omega}\big(\sigma(\rot_gf,\gt)+\langle\curl_g\sigma,\gt\rangle\,f\big)\,\vo{\d \Omega},
\end{align*}
i.e., $\inc_g$ and $\rot\rot_g$ are $\Ltwo$-adjoint with respect to the $g$-weighted $\Ltwo$ inner product.

Above, we used the $g$ inner product $\langle\cdot,\cdot\rangle:=g(\cdot,\cdot)$ extended from vector fields to arbitrary order tensors $\TT^l_k(\Omega)$, e.g.
\begin{align*}
	\langle\curl_g\sigma,X\rangle = (\curl_g\sigma)(X)\qquad \forall \sigma\in \TT^2_0(\Omega),\,X\in\Xm{\Omega}.
\end{align*}

The following definition of the distributional covariant incompatibility operator has been derived in~\cite[Proposition 4.6]{GNSW2023}. A similar expression for  the vertex contributions can be found in \cite{Chr2024}.

\begin{definition}
	Let $g\in\Regge^+(\T)$ and $u\in\VV(\T)$. The distributional incompatibility operator $\dincop:\Regge(\T)\to \VV(\T)^\prime$ is defined by 
	\begin{align}
		\label{eq:distr_cov_inc}
		(\dincop\sigma)(u) = \sum_{T\in\T}\bigg[\int_T \inc_g(\sigma)\,u\,\vo{T}-\int_{\d T}u\,\langle\curl_g\sigma+d(\sigma_{\gn\gt}),\gt\rangle\,\vo{\d T}+\sum_{V\in \V_T}\jmp{\sigma_{\gn\gt}}^T_V\,u(V)\bigg],
	\end{align}
where $\V_T=\{V\in\V\,:\, V\in T\}$ and, cf. e.g. \cite{GNSW2023b},
\[
\jmp{\sigma_{\gn\gt}}^T_V= \big(\sigma|_T(\gn^T_{E_+},\gt_V^{E_+})+\sigma|_T(\gn^T_{E_-},\gt_V^{E_-})\big)(V).
\]
The distributional covariant $\rot\rot$ operator $\drotrotop:\VV(\T)\to \Regge(\T)^\prime$ is defined by 
	\begin{align}
		\label{eq:distr_cov_rotrot}
		(\drotrotop u)(\sigma)=\sum_{T\in\T}\bigg[\int_T \langle \rot\rot_gu,\sigma\rangle\,\vo{T}+\int_{\partial T} \sigma_{\gt\gt}\,\langle\nabla_g u,\gn\rangle\,\vo{\partial T}\bigg].
	\end{align}
  Note that
  one of the boundary terms in~\eqref{eq:distr_cov_inc} admits a representation using  the geodesic curvature $\GeodCurv_{\gn}$, namely
	\begin{align*}
		d(\sigma_{\gn\gt})(\gt) = \nabla_{\gt}\big(\sigma(\gn,\gt)\big) = (\nabla_{\gt}\sigma)(\gn,\gt) + (\sigma(\gn,\gn)-\sigma(\gt,\gt))\,\GeodCurv_{\gn}.
	\end{align*}
\end{definition}

Next we show that the (distributional) adjoint of $\inc_g$ is  $\rot\rot_g$ in the following sense.
\begin{lemma}
	\label{lem:inc_rotrot_adjoint}
	Let $\sigma\in\Regge(\T)$ and $u\in \VV(\T)$. There holds
\begin{align*}
	(\dincop\sigma)(u) = (\drotrotop u)(\sigma).
\end{align*}
\end{lemma}
\begin{proof}
  This follows by integration by parts on each $T\in\T$:
	\begin{align*}
		(\dincop\sigma)(u) &=\sum_{T\in\T}\Big[\int_T \inc_g(\sigma)\,u\,\vo{T}-\int_{\d T}u\,\langle\curl_g\sigma+d(\sigma_{\gn\gt}),\gt\rangle\,\vo{\d T}+\sum_{V\in \V_T}\jmp{\sigma_{\gn\gt}}^T_V\,u(V)\Big]\\
		&=\sum_{T\in\T}\Big[\int_T \inc_g(\sigma)\,u\,\vo{T}-\int_{\d T}\big(u\,\langle\curl_g\sigma,\gt\rangle-\sigma_{\gn\gt}\nabla_{\gt}u\big)\,\vo{\d T}\Big]\\
		&=\sum_{T\in\T}\Big[\int_T \langle\curl_g\sigma,\rot_gu\rangle\,\vo{T}+\int_{\d T}\sigma_{\gn\gt}\nabla_{\gt}u\,\vo{\d T}\Big]\\
		&=\sum_{T\in\T}\Big[\int_T \langle\sigma,\rot\rot_gu\rangle\,\vo{T}+\int_{\d T}\big(\sigma(\gt,\rot_g u)+\sigma_{\gn\gt}\nabla_{\gt}u\big)\,\vo{\d T}\Big]\\
		&=\sum_{T\in\T}\Big[\int_T \langle\sigma,\rot\rot_gu\rangle\,\vo{T}+\int_{\d T}\sigma_{\gt\gt}\nabla_{\gn}u\,\vo{\d T}\Big] = (\drotrotop u)(\sigma).
	\end{align*}
\end{proof}

In \cite{GNSW2023} we proved an integral representation of the densitized Gauss curvature using a parametrization starting from the Euclidean metric $\Eucl$
\begin{align*}
	\Kog(g)(u)=-\frac{1}{2}\int_0^1b(\Eucl+t(g-\Eucl);g-\Eucl,u)\,dt,\quad \text{ with } b(g;\sigma,u)= (\dincop\sigma)(u)
\end{align*}
and used its integrand to derive convergence results. In this work, we follow the approach of \cite{GN2023,GN2023b,GNSW2023b} and consider directly the integral representation of the error as follows. Let $\gpar(t)= \gex + t(\gappr-\gex)$ and $\sigma=\gpar^\prime(t) = \gappr-\gex$. Then there holds the integral representation of the error
\begin{align}
	\label{eq:int_repr_error}
	\big(\Kog(\gappr)-\Kex\volformex\big)(u) = -\frac{1}{2}\int_0^1(\dincop[\gpar(t)]\sigma)(u)\,dt.
\end{align}
To derive error estimates, one important part will be analyzing the integrand of \eqref{eq:int_repr_error}, or, more precisely, its adjoint $(\dincop[\gpar(t)]\sigma)(u)=(\drotrotop[\gpar(t)] u)(\sigma)$.

\section{Error analysis}
\label{sec:error_analysis}

In this section we prove a priori error estimates for the lifted densitized Gauss curvature $\Gaussh\volformh$ and the Gauss curvature $\Gaussh$. First, we consider the $\Hmone$-norm as basis and then show estimates also for the stronger Sobolev norms $\Ltwo$ and $H^r$, $r\geq 1$. Let $\Omega\subset\R^2$ be a domain with a given exact metric tensor $\gex$ and corresponding exact Gauss curvature $\Kex=\Gauss(\gex)$. For simplicity, we assume in this section that homogeneous Dirichlet data $\Kex^D=0$ is described on the whole boundary, $\Gamma_D=\d\Omega$.

\subsection{Statement of main theorem}

We consider a sequence of quasiuniform
(hence shape-regular)
affine-equivalent triangulations $\{\T_h\}_{h>0}$ with maximal mesh-size $h=\max_{T\in\T_h}h_T$, where $h_T=\mathrm{diam}(T)$. On the triangulations a sequence of Regge metrics $\{\gappr\}_{h>0}$ with $\gappr\in\Regge_h^k$, $k\geq 0$ (defined in \eqref{eq:ReggeFE}) is given. To be precise, we assume that $\gappr$ is the canonical interpolant of $\gex$. This interpolant \cite{li18}, denoted by
$\RegInt[k]:W^{s,p}(\Omega,\Sc)\to\Regge_h^k$, $p\in[1,\infty]$, $s\in (1/p,\infty]$, satisfies the following
equations
\begin{subequations}
	\label{eq:RegInt}  
	\begin{align}  
		\int_E(\RegInt[k]\sigma)_{\tv\tv}\,q\,\dl
		& = \int_E\sigma_{\tv\tv}\,q\,\dl
		&& \text{for all } q\in \Pol^{k}(E)
		\text{ and edges $E$ of $\d T$},\label{eq:RefInt_edge}
		\\
		\int_T\RegInt[k]\sigma:\rho\,\da
		& = \int_T\sigma:\rho\,\da
		&& \text{for all } \rho\in \Pol^{k-1}(T,\R^{2 \times 2}), \,T\in\T.
		\label{eq:RefInt_trig}
	\end{align}
\end{subequations}
Equations \eqref{eq:RegInt} can be interpreted as orthogonality requirements, preserving specific moments at edges and elements. Note that when $\rho$ is a skew-symmetric matrix, both sides of~\eqref{eq:RefInt_trig} vanish, so~\eqref{eq:RefInt_trig} is nontrivial only for symmetric $\rho$. 

Throughout, we use standard
Sobolev spaces $\Wsp[\Omega]$ and their norms and seminorms for any $s\geq 0$ and
$p\in [1,\infty].$ When the domain is $\Omega$, we omit it from the
norm notation if there is no chance of confusion.  We also use the
elementwise norms $\|u\|_{\Wsph}^p =\sum_{T\in \T_h}\|u\|^p_{\Wsp[T]},$
with the usual adaption for $p=\infty$. When $p=2$, we put
$\|\cdot\|_{\Hsh}=\|\cdot\|_{W^{s,2}_h}$.  Let $D\subset\Omega$ and define
\begin{align}
	\label{eq:def_norms}
\nrm{\sigma}_{2,D} = \| \sigma \|_{L^2(D)} + h \| \sigma \|_{H_h^1(D)}.
\end{align}
If $D$ is the whole domain $\Omega$, we neglect the subscript in \eqref{eq:def_norms}.

We write $a\lesssim b$ if there exists a mesh-size independent
constant $C>0$
which may depend on---unless otherwise stated---the domain $\Omega$, the polynomial degree $k$,
    the shape regularity constant $\sigma(\T_h)$ of $\T_h$, the $W^{2,\infty}$-norm of $\gex$, $\Linf$-norm of $\gex^{-1}$, and the $\Hone$-norm of $\Kex$ i.e.
\begin{align}
	\label{eq:dependencies_C}
	C=C(\Omega,k,\sigma(\T_h),\|\gex\|_{W^{2,\infty}},\|\gex^{-1}\|_{\Linf},\|\Kex\|_{\Hone}).
\end{align}
We abbreviate the $\Ltwo$-inner product of two scalar functions and the $g$-weighted inner product by
\[
(u,v)_{\Ltwo}:=\int_{\Omega}uv\,\da,\qquad (u,v)_{\Ltwo,g}:=\int_{\Omega}uv\,\sqrt{\det g}\,\da,\qquad\qquad  u,v\in\Ltwo[\Omega].
\]

Our main theorem reads as follows:
\begin{theorem}
	\label{thm:improved_rate_Hm1}
	Let $k\geq1$ be an integer, $\{\T_h\}_{h>0}$ a sequence of quasiuniform 
        triangulations,  $\{\gappr\}_{h>0}$ a sequence of metric approximations $\gappr=\RegInt[k]\gex$ with $\gex\in W^{2,\infty}(\Omega,\Sc),$
        so that  $\Kex\in\Ltwo[\Omega]$,
        $\volformh=\volform(\gappr)$, and $\Gaussh\in\Vo_h^k$ the lifted distributional Gauss curvature from~\eqref{eq:lifted_distr_Gauss}.
        Suppose also that $\Kex=0$ on the boundary $\d \Omega$.
        Then there exists an $h_0>0$ such that for all $h\leq h_0$
	\begin{align*}
		\|\Gaussh\volformh-\Kex\volformex\|_{\Hmone}&\leq C h\big(\nrm{\gappr-\gex}_2+\inf\limits_{v_h\in\Vo_h^k}\|v_h-\Kex\|_{\Ltwo[]}+\|\gappr-\gex\|_{\Linf}\big),
	\end{align*}
	where the constant $C$ depends on $\Omega$, the shape regularity, polynomial degree $k$, $\|\gex\|_{W^{2,\infty}}$, and $\|\gex^{-1}\|_{\Linf}$. If additionally for $0\leq l\leq k+1$, $\gex\in W^{l,\infty}(\Omega,\Sc)$ and $\Kex\in H^{l}(\Omega)$, then
	\begin{align*}
		\|\Gaussh\volformh-\Kex\volformex\|_{\Hmone}&\leq C h^{l+1}\big(\|\gex\|_{W^{l,\infty}}+\|\Kex\|_{H^{l}}\big).
	\end{align*}
\end{theorem}

Further
convergence results in stronger Sobolev norms
follow.

\begin{corollary}
	\label{cor:improved_conv_dens_curvature}
	Under the assumptions of Theorem~\ref{thm:improved_rate_Hm1}, there holds for  $0\leq l \leq k+1$, $0\leq r\leq l$
	\begin{align*}
			\|\Gaussh\volformh-\Kex\volformex\|_{H^r_h}&\leq C h^{-r}\big(\|\gappr-\gex\|_{L^\infty}+\nrm{\gappr-\gex}_2+\inf\limits_{v_h\in\Vo_h^{k}}\|v_h-\Kex\|_{\Ltwo}\\
			&\qquad\qquad+h^{l}\|\Kex\|_{H^{l}}+\inf\limits_{v_h\in\Vo_h^{k}}\|v_h-\Kex\volformex\|_{\Ltwo}+h^{l}\|\Kex\volformex\|_{H^{l}}\big)\\
			&\leq C h^{l-r}\big(\|\gex\|_{W^{l,\infty}} + \|\Kex\volformex\|_{H^{l}}+ \|\Kex\|_{H^{l}}\big),
	\end{align*}
where the constant $C>0$ depends additionally on $\|\Kex\|_{\Hone}$.
\end{corollary}

\begin{remark}[Convergence of pure Gauss curvature]
  \label{rem:pureK}
  In  contrast to the densitized Gauss curvature, we refer to the $\Kex$ (without multiplication by  the volume form) as the ``pure Gauss curvature.'' 
  For the error in the pure Gauss curvature, $\|\Gaussh-\Kex\|_{H^r_h}$, $-1\leq r \leq k$, the same convergence rates as proved in Theorem~\ref{thm:improved_rate_Hm1} and Corollary~\ref{cor:improved_conv_dens_curvature} are obtained: see Theorem~\ref{thm:improved_rate_Hm1_pure} and Corollary~\ref{cor:improved_conv_curvature} in Section~\ref{subsec:ana_pure_Gauss}.
\end{remark}
\begin{remark}[Optimal convergence]
	If we insert $l=k+1$ in Theorem~\ref{thm:improved_rate_Hm1}, we obtain the convergence rate $\mathcal{O}(h^{k+2})$, which is of one order higher than $\mathcal{O}(h^{k+1})$ proved in \cite[Theorem 6.5]{GNSW2023} and two orders higher compared to \cite[Theorem 4.1]{Gaw20}. 
	Furthermore, using $l=k+1$ and $r=0$ in Corollary~\ref{cor:improved_conv_dens_curvature} yields an $\Ltwo$ convergence rate of  $\mathcal{O}(h^{k+1})$, which is the ``optimal'' in the sense that it is the rate of convergence of the $L^2$~best approximation  from  $\Vo^k_h$. The requirement of at least linear elements, $k\geq 1$, cannot be relaxed to $k=0$ as for $\gappr\in\Regge_h^0$ there is no (non-trivial) Lagrange finite element function in $\Vo_h^0$. In \cite{GNSW2023} we observed that the pairing of the lowest order elements $\gappr\in\Regge_h^0$ and $\Gaussh\in\VV_h^1$ does not lead to an improved $\Ltwo$-convergence rate of $\mathcal{O}(h)$. In fact, our numerical examples  in \cite{GNSW2023} showed that we may expect no convergence in the $\Ltwo$-norm  in general in this case.
\end{remark}

\subsection{Basic estimates}

We need a number of preliminary estimates to proceed with our
analysis.  The approximation properties of the Regge elements are well
understood.  By the Bramble-Hilbert lemma, on any $T \in \T$, see \cite[Theorem 2.5]{li18},
	\begin{align}
		&\|(\idop-\RegInt[k])\sigma\|_{W^{r,p}(T)}\leq C h^{l-r} |\sigma|_{W^{l,p}(T)},
		\label{eq:reg_appr_lo_vol}
	\end{align}
	for $p\in [1,\infty]$, $l\in(1/p,k+1]$, $r\in [0,l]$, $\sigma\in W^{l,p}(T,\Sc)$, and $C$ depends on $k$, $r$, $l$, and the shape regularity $\sigma(T)$ of $T$. A similar estimate holds for the elementwise $L^2(T)$-projection into the space of polynomials of order $k$, which we denote by $\LtwoInt[k]$, see e.g. \cite[Theorem 4.4.4]{BS2008},
	\begin{align*}
		&\|(\idop-\LtwoInt[k])f\|_{\Lp[T]}\leq h^{l} C|f|_{W^{l,p}(T)}
	\end{align*}
for $l\in(1/p,k+1]$, $f\in W^{l,p}(T)$, and $C$ depends on  $k$, $l$, and the shape regularity $\sigma(T)$ of $T$. The same holds if we replace $T$ by an edge $E\in\E$ and the edge-wise $\Ltwo$-projection denoted by $\LtwoInt[E,k]$. 

Let $E\subset \d T$ be an edge of $T$. We also need the following well-known estimates that follow
from scaling arguments: for all $u\in \Hone[T]$
\begin{align}
	\|u\|_{\Ltwo[E]}^2\,\lesssim\, 
	h^{-1}\|u\|_{\Ltwo[T]}^2 +h\|\nabla u\|_{\Ltwo[T]}^2
	\label{eq:trace_inequ}
\end{align}
and for all $u\in \Pol^k(T)$,
\begin{align}
	&\|u\|_{\Ltwo[E]}\,
	\lesssim h^{-1/2}\|u\|_{\Ltwo[T]},\qquad |u|_{H^l(T)}\,\lesssim\,
	h^{-l}\|u\|_{\Ltwo[T]},\quad 1\leq l\leq k.
	\label{eq:discr_inverse_inequ}
\end{align}

The $\Ltwo$-orthogonal projection with respect to $\gex$ into Lagrange elements $\LtwoProj[k]:\Ltwo(\Omega)\to\Vo_h^k$ is defined via its orthogonality property
\begin{align}
	\label{eq:proj_l2_orth}
	\int_{\Omega}\big(\LtwoProj[k]u-u\big)\,v_h\,\volformex=0,\qquad \forall v_h\in\Vo_h^k.
\end{align}
It has the following well-known
stability and approximation properties on quasiuniform meshes, see e.g.~\cite{CrouzThome87} or
\cite[Lemma 4.7]{Gaw20},
\begin{subequations}
	\begin{align}
	&\|\LtwoProj[k]u\|_{\Ltwo}\lesssim \|u\|_{\Ltwo},&&\quad \forall u\in\Ltwo[\Omega],\label{eq:proj_l2_stab}\\
	&\|\LtwoProj[k]u\|_{\Hone}\lesssim \|u\|_{\Hone},&&\quad \forall u\in\Honez[\Omega],\label{eq:proj_h1_stab}\\
	&\|\LtwoProj[k]u-u\|_{\Ltwo}\lesssim \inf\limits_{u_h\in \Vo_h^k}\|u_h-u\|_{\Ltwo}\lesssim h\,\|u\|_{\Hone}, &&\quad \forall u\in\Honez[\Omega].\label{eq:proj_approx}
\end{align}
\end{subequations}

Since $\gappr = \RegInt[k] \gex$ approaches $\gex$ as $h \to 0$, we tacitly assume throughout that $h$ has become sufficiently small ($h\le h_0$) to guarantee that the approximated metric $\gappr$ is positive definite throughout. Further, thanks to \eqref{eq:reg_appr_lo_vol} (with $p=r=l=2$ and $k\geq 1$) and \eqref{eq:dependencies_C} we have that $\sup_{T\in\T_h}\|\gappr\|_{W^{2,\infty}_h(T)}\leq C$. The following estimates are a consequence of \cite{Gaw20,GN2023,GN2023b}: for $p\in [1,\infty]$, $t\in[0,1]$, $\gpar(t)=\gex+t(\gappr-\gex)$, $l \in\{ 0, 1,2\}$,
\begin{subequations}
	\begin{align}
	&\|\gpar(t)-\gex\|_{W^{l,p}_h}+\|\gpar^{-1}(t)-\gex^{-1}\|_{W^{l,p}_h}+\|\sqrt{\det {\gpar(t)}}-\sqrt{\det \gex}\|_{W_h^{l,p}}
	\,\lesssim \,
	\|\gappr-\gex\|_{W_h^{l,p}},\label{eq:est_inv_by_ten}
	\\
	&\|{\gpar(t)}\|_{W^{2,\infty}_h}+\|{\gpar(t)}^{-1}\|_{\Linf} + \|\sqrt{\det {\gpar(t)}}\|_{\Linf}+\|\sqrt{\det {\gpar(t)^{-1}}}\|_{\Linf}
	\,\lesssim\, 1.  \label{eq:bound_gt}
\end{align}
\end{subequations}
Further, for all $x$ in the interior of any element $T\in \T$ and for all
$u \in \R^2$, as well as for the $\Ltwo$ inner product there holds the following equivalences
\begin{equation*}
	u^\trans u\,\lesssim\, u^\trans {\gpar(t)}(x)u
	\,\lesssim\, u^\trans u,\qquad\qquad (\cdot,\cdot)_{\Ltwo}\lesssim (\cdot,\cdot)_{\Ltwo,\gpar(t)}\lesssim(\cdot,\cdot)_{\Ltwo}.
\end{equation*}

\subsection{Analysis of distributional rotrot operator}
In this section, we derive improved convergence rates of the distributional covariant $\rot\rot_g$ operator \eqref{eq:distr_cov_rotrot}. The proof strategy follows \cite[Theorem 6.1, (6.3)]{GNSW2023}, however, adapted from the distributional covariant $\curl$ to the $\rot\rot$ operator.
\begin{proposition}
	\label{prop:conv_rotrot}
	Let $k\geq 1$ be an integer, $g\in\Regge^+(\T)$, $\rho\in W^{s,p}(\Omega,\Sc)$, $p\in [1,\infty]$, $s\in (1/p,\infty]$, $\rho_h=\RegInt[k]\rho$, and $u_h\in \Vo_h^k$.  Then there holds
	\begin{align*}
		|(\drotrotop u_h)(\rho-\rho_h)|\leq C h\nrm{\rho-\rho_h}_2\|u_h\|_{\Hone},
	\end{align*} 
where the constant $C>0$ depends on $\Omega$, the mesh regularity, $k$, $\|g\|_{W^{2,\infty}_h}$, and $\|g^{-1}\|_{\Linf}$.
\end{proposition}
\begin{proof}
	First, consider the element terms of \eqref{eq:distr_cov_rotrot}. Comparing with coordinate expressions \eqref{eq:cov_rotrot_coo} and \eqref{eq:coo_volforms}
	we can find smooth functions 
	\begin{align*}
		F(g)=\frac{1}{\sqrt{\det g}},\qquad \mt{G(g)}^{ij}_k=\frac{1}{\sqrt{\det g}}\veps^{jq}\,\left(\Gamma_{qk}^i(g)-\Gamma_{lq}^l(g)\,\delta_k^i\right),
	\end{align*} 
	such that
\begin{align*}
	\int_T \langle \rot\rot_g u_h,\rho-\rho_h\rangle\,\vo{T} &= \int_T \mt{\rho-\rho_h}_{ij}\left(F(g)\mt{\rot\rot u_h}^{ij}+\mt{G(g)}_k^{ij}\mt{\rot u_h}^k\right)\,\da\\
	& = \int_T \mt{\rho-\rho_h}_{ij}\mt{\rot\rot u_h}^{ij}\left((\LtwoInt[1]+\LtwoInt[1,\perp])F(g)\right) \\
	&\qquad+ \mt{\rho-\rho_h}_{ij}\mt{\rot u_h}^{k}\left( (\LtwoInt[0]+\LtwoInt[0,\perp])\mt{G(g)}^{ij}_k\right)\,\da\\
                                                                                   & = \int_T \mt{\rho-\rho_h}_{ij}\mt{\rot\rot u_h}^{ij}\LtwoInt[1,\perp]\left(F(g)\right)
  \\
  & \qquad + \mt{\rho-\rho_h}_{ij}\mt{\rot u_h}^k \LtwoInt[0,\perp]\left(\mt{G(g)}_k^{ij}\right)\,\da\\
	&\leq C h_T^2\|\rho-\rho_h\|_{\Ltwo[T]}|u_h|_{\Htwo[T]} + h_T\|\rho-\rho_h\|_{\Ltwo[T]}\|u_h\|_{\Hone[T]}\\
	&\leq C h_T\|\rho-\rho_h\|_{\Ltwo[T]}\|u_h\|_{\Hone[T]}.
\end{align*}
Above, we split the nonlinear terms using $\Ltwo$-projections $\LtwoInt[k]$ and co-projections $\LtwoInt[k,\perp]$ onto elementwise polynomials, and used their approximation property. Further, we exploited that the first $k-1$ moments of $\rho-\rho_h$ are zero \eqref{eq:RefInt_trig} and used inverse inequality \eqref{eq:discr_inverse_inequ}.

Next, we focus on the element-boundary terms of \eqref{eq:distr_cov_rotrot}. With the coordinate expressions for the $g$-normalized tangent and normal vector, see e.g. \cite{GNSW2023} with the Euclidean vectors $(\tv,\nv)$ and the notation $g^{\nv\nv}=g^{ij}\nv_i\nv_j$
\[
\gt^i = \frac{1}{\sqrt{g_{\tv\tv}}}\tv^i,\qquad \gn^i = \frac{g^{ij}\nv_j}{\sqrt{g^{\nv\nv}}}
\] we collect all terms depending on $g$ in the nonlinear function 
\begin{align*}
	\mt{H(g)}^i= \frac{1}{\sqrt{g_{\tv\tv}\,g^{\nv\nv}}}g^{ij}\nv_j.	
\end{align*}
 We split $H$ with the edgewise $\Ltwo$-interpolant $\LtwoInt[E,k]$. Then, we use H\"older inequality, as well as the trace inequalities \eqref{eq:trace_inequ} and \eqref{eq:discr_inverse_inequ} to obtain
\begin{align*}
	\int_{\d T} (\rho-\rho_h)(\gt,\gt)\langle \nabla u_h,\gn\rangle\,\vo{\d T} &= \int_{\d T}(\rho-\rho_h)(\tv,\tv) \d_iu_h\,\left((\LtwoInt[E,1]+\LtwoInt[E,1,\perp])\mt{H(g)}^i\right)\,\dl\\
	&\leq C h_T^2\, \|\rho-\rho_h\|_{\Ltwo[\d T]}\|\nabla u_h\|_{\Ltwo[\d T]}\|g\|_{W^{2,\infty}(\d T)}\\
	&\leq  C h_T \nrm{\rho-\rho_h}_{2,T}\|u_h\|_{\Hone[T]}.
\end{align*}
Summing over all elements $T\in\T$ finishes the proof.
\end{proof}

Due to Lemma~\ref{lem:inc_rotrot_adjoint} we obtain as a byproduct the convergence of the distributional covariant incompatibility operator.
\begin{corollary}
	Under the assumptions of Proposition~\ref{prop:conv_rotrot}, there holds
	\begin{align*}
		\left|\left(\dincop(\sigma-\sigma_h)\right)(u_h)\right|\leq C h\nrm{\sigma-\sigma_h}_2\|u_h\|_{\Hone},
	\end{align*} 
	where the constant $C>0$ depends on $\Omega$, the mesh regularity, $k$, $\|g\|_{W^{2,\infty}_h}$, and $\|g^{-1}\|_{\Linf}$.
\end{corollary}

\subsection{Proof of Theorem~\ref{thm:improved_rate_Hm1}}
We are now in position to prove our main theorem. The proof strategy is inspired by the proofs of \cite[Theorem 6.5]{GNSW2023} and \cite[Theorem 4.1]{Gaw20}.
\begin{proof}[Proof of Theorem~\ref{thm:improved_rate_Hm1}]
	We start with the definition of the $\Hmone$-norm noting that $\Gaussh\volformh$ and $\Kex\volformex$ are square integrable
	\begin{align*}
		\|\Gaussh\volformh-\Kex\volformex\|_{\Hmone}=\sup\limits_{u\in\Honez[\Omega]}\frac{(\Gaussh\volformh-\Kex\volformex,u)_{\Ltwo}}{\|u\|_{\Hone}}.
	\end{align*}
Next, we add and subtract the $\Ltwo$-orthogonal interpolant \eqref{eq:proj_l2_orth} $u_h:= \LtwoProj[k] u$, $\LtwoProj[k]:\Ltwo(\Omega)\to \Vo_h^k$, to split the error into three parts
\begin{align*}
	(\Gaussh\volformh-\Kex\volformex,u)_{\Ltwo} &= (\Gaussh\volformh-\Kex\volformex,u_h)_{\Ltwo} + (\Gaussh\volformh-\Kex\volformex,u-u_h)_{\Ltwo}\\
	&=(\Gaussh\volformh-\Kex\volformex,u_h)_{\Ltwo}+ (\Gaussh-\Kex,u-u_h)_{{\Ltwo},\gex}+(\Gaussh(\volformh-\volformex),u-u_h)_{\Ltwo}\\
	&=:s_1+s_2+s_3.
\end{align*}
We use \eqref{eq:lifted_distr_Gauss}, the integral representation of the error \eqref{eq:int_repr_error} with $\gpar(t)= \gex + t(\gappr-\gex)$ and $\sigma=\gpar^\prime(t) = \gappr-\gex$, and the adjoint of the distributional incompatibility operator Lemma~\ref{lem:inc_rotrot_adjoint}
\begin{align*}
	s_1&= \left(\Kog(\gappr)-\Kex\volformex\right)(u_h) = -\frac{1}{2}\int_0^1 \left(\dincop[\gpar(t)]\sigma\right)(u_h)\,dt = -\frac{1}{2}\int_0^1\left(\drotrotop[\gpar(t)]u_h\right)(\sigma)\,dt.
\end{align*}
From Proposition~\ref{prop:conv_rotrot} (setting $g=\gpar(t)$ and $\rho=-\gex$) we obtain together with the $\Hone$-stability \eqref{eq:proj_h1_stab}
\begin{align}
	 \label{eq:estimate_s1}
	|s_1|\lesssim h\nrm{\sigma}_2\|u_h\|_{\Hone} \lesssim h\nrm{\gappr-\gex}_2\|u\|_{\Hone}.
\end{align}
For $s_2$ we use the definition of the $\Ltwo$-orthogonal interpolant \eqref{eq:proj_l2_orth} $u_h=\LtwoProj[k]u$, Cauchy-Schwarz inequality, and the approximation property \eqref{eq:proj_approx} of $u_h$. For arbitrary $v_h\in \Vo_h^k$ there holds
\begin{align*}
	(\Gaussh-\Kex,u-u_h)_{{\Ltwo},\gex}&=(v_h-\Kex,u-u_h)_{{\Ltwo},\gex}\lesssim \|v_h-\Kex\|_{\Ltwo}\|u-u_h\|_{\Ltwo}\lesssim h\|v_h-\Kex\|_{\Ltwo}\|u\|_{\Hone}
\end{align*}
and thus,
\begin{align*}
	|s_2|\lesssim h\inf\limits_{v_h\in\Vo_h^k}\|v_h-\Kex\|_{\Ltwo}\|u\|_{\Hone}.
\end{align*}

Before we turn to the third term $s_3$ we show that the lifted Gauss curvature $\Gaussh$ is bounded in the $\Ltwo$-norm by using $\Gaussh\in\Vo_h^k$ instead of $u_h$ in estimate \eqref{eq:estimate_s1} 
\begin{align*}
	\|\Gaussh\|^2_{\Ltwo}&\lesssim (\Gaussh\volformh,\Gaussh)_{\Ltwo}=(\Gaussh\volformh-\Kex\volformex,\Gaussh)_{\Ltwo}+ (\Kex,\Gaussh)_{\Ltwo,\gex}\\
	&\lesssim h\|\gappr-\gex\|_{\Ltwo}\|\Gaussh\|_{\Hone}+\|\Kex\|_{\Ltwo}\|\Gaussh\|_{\Ltwo}\\
	&\lesssim \big(\|\gappr-\gex\|_{\Ltwo}+\|\Kex\|_{\Ltwo}\big)\|\Gaussh\|_{\Ltwo}.
\end{align*}
For the last inequality we used the inverse estimate \eqref{eq:discr_inverse_inequ}. Dividing by $\|\Gaussh\|_{\Ltwo}$ yields the boundedness of $\Gaussh$.

Using H\"older inequality, the approximation property \eqref{eq:proj_approx} of $u_h$, inequality \eqref{eq:est_inv_by_ten}, and that $\|\Gaussh\|_{\Ltwo}\lesssim 1$ yields the following estimate for $s_3$
\begin{align*}
	|s_3|\leq \|\Gaussh\|_{\Ltwo}\|\volformh-\volformex\|_{\Linf}\|u-u_h\|_{\Ltwo} \lesssim h\|\gappr-\gex\|_{\Linf}\|u\|_{\Hone}.
\end{align*}
Combining all results yields 
\begin{align*}
  \|\Gaussh\volformh-\Kex\volformex\|_{\Hmone}
&=\sup\limits_{u\in\Honez[\Omega]}\frac{(\Gaussh\volformh-\Kex\volformex,u)_{\Ltwo}}{\|u\|_{\Hone}}
  \\
  &\lesssim h(\nrm{\gappr-\gex}_2+\inf\limits_{v_h\in\Vo_h^k}\|v_h-\Kex\|_{\Ltwo}+\|\gappr-\gex\|_{\Linf}).
\end{align*}
Using standard interpolation techniques we obtain the desired convergence rate for $0\leq l\leq k+1$
\begin{align*}
	\|\Gaussh\volformh-\Gauss\volform\|_{\Hmone}\lesssim h^{l+1}\,\big(\|\gex\|_{W^{l,\infty}}+\|\Kex\|_{H^{l}}\big).
\end{align*}
\end{proof}

\subsection{Analysis of the lifting of pure Gauss curvature}
\label{subsec:ana_pure_Gauss}
To relate the error of the Gauss curvature with the densitized Gauss curvature we need the following result.
\begin{lemma}
	Under the assumptions of Theorem~\ref{thm:improved_rate_Hm1}, there holds for $0\leq l\leq k+1$
	\label{lem:approx_diff_volforms}
	\begin{align*}
			\| \Gaussh(\volformh-\volformex)\|_{\Hmone}&\lesssim h\,\|\gappr-\gex\|_{\Linf}\|\Gaussh\|_{\Honeh}\\
		&\lesssim h^{l+1}\|\gex\|_{W^{l,\infty}}\|\Gaussh\|_{\Honeh}.
	\end{align*}
\end{lemma}
\begin{proof}
	By noting that $\volformh-\volformex = \mt{\gappr-\gex}_{ij} \mt{F(\gappr,\gex)}^{ij}$ with the smooth function
	\begin{align*}
	\mt{F(\gappr,\gex)} = \frac{1}{\sqrt{\det \gappr}+\sqrt{\det \gex}}\begin{bmatrix}
		(\gappr)_{22} & \frac{1}{2}(\gex+\gappr)_{12}\\
		\frac{1}{2}(\gex+\gappr)_{12} & \gex_{11}
	\end{bmatrix},
	\end{align*} 
	and using that for $k\geq1$ the constant moment of the difference $\gappr-\gex$ is zero due to \eqref{eq:RefInt_trig}, we obtain
\begin{align*}
	(\Gaussh(\volformh-\volformex),u)_{\Ltwo}&=\int_{\Omega}\mt{\gappr-\gex}_{ij} \mt{F(\gappr,\gex)}^{ij}\Gaussh u\,\da \\
	&=\int_{\Omega}\mt{\gappr-\gex}_{ij} \left((\LtwoInt[0]+\LtwoInt[0,\perp])\left(\mt{F(\gappr,\gex)}^{ij}\Gaussh u\right)\right)\,\da\\
	&=\int_{\Omega}\mt{\gappr-\gex}_{ij} \LtwoInt[0,\perp]\left(\mt{F(\gappr,\gex)}^{ij}\Gaussh u\right)\,\da\\
	&\leq \|\gappr-\gex\|_{\Linf}\|\LtwoInt[0,\perp]\left(F(\gappr,\gex)\Gaussh u\right)\|_{L^1}\\
	&\lesssim h \|\gappr-\gex\|_{\Linf} \|\Gaussh\|_{\Honeh}\|u\|_{\Hone},
\end{align*}
finishing the proof.
\end{proof}

By Lemma~\ref{lem:approx_diff_volforms}, inverse estimate \eqref{eq:discr_inverse_inequ}, and boundedness of $\Gaussh$ in $\Ltwo$, $\|\Gaussh\|_{\Honeh}\lesssim h^{-1}\|\Gaussh\|_{\Ltwo}\lesssim h^{-1}$, we can deduce a suboptimal convergence rate of the error of the pure lifted Gauss curvature
\begin{align}
	\label{eq:error_lift_curv1}
	\begin{split}
		\|\Gaussh-\Kex\|_{\Hmone}&\lesssim \|\Gaussh\volformex-\Kex\volformex\|_{\Hmone}\\
		&\leq  \|\Gaussh\volformh-\Kex\volformex\|_{\Hmone} +\|\Gaussh(\volformex-\volformh)\|_{\Hmone}\\
		&\lesssim h\big(\nrm{\gappr-\gex}_2+\inf\limits_{v_h\in\Vo_h^k}\|v_h-\Kex\|_{\Ltwo}+\|\gappr-\gex\|_{\Linf}\big)+h\,\|\gappr-\gex\|_{\Linf}\|\Gaussh\|_{\Honeh}\\
		&\lesssim h^{l+1}\big(\|\gex\|_{W^{l+1,\infty}}+\|\Kex\|_{H^l}\big),\qquad\qquad 0\leq l\leq k \text{ (not $k+1$)}.
	\end{split}
\end{align}
To correct the convergence of the lifted Gauss curvature and to prove optimal rates for the (densitized) lifted Gauss curvature in stronger Sobolev norms we consider a bootstrapping-like technique. First, we can easily adapt the proof of \cite[p. 1818]{Gaw20} and \cite[Corollary 6.6]{GNSW2023} to deduce for $0\leq l \leq k$ and $0\leq r\leq l$
\begin{align*}
		\|\Gaussh-\Kex\|_{H^r_h}&\lesssim h^{-r}\big(\nrm{\gappr-\gex}_2+h^{-1}\|\gappr-\gex\|_{\Linf}+\inf\limits_{v_h\in\Vo_h^k}\|v_h-\Gauss\|_{\Ltwo}+h^l|\Kex|_{H^l}\big)\\
	&\lesssim h^{l-r}\big(\|\gex\|_{W^{l+1,\infty}}+\|\Kex\|_{H^l}\big)
\end{align*}
and therefore the boundedness in the elementwise $\Hone$-norm for $l=r=1$ and $k\geq 1$, 
$$\|\Gaussh\|_{\Honeh}\leq \|\Gaussh-\Kex\|_{\Honeh} + \|\Kex\|_{\Hone} \lesssim \|\gex\|_{W^{2,\infty}}+\|\gex\|_{\Hone}+\|\Kex\|_{\Hone}\lesssim 1.$$
Thus, instead of using the inverse inequality in \eqref{eq:error_lift_curv1} we directly obtain the improved convergence rate:

\begin{theorem}
	\label{thm:improved_rate_Hm1_pure}
	Under the assumptions of Theorem~\ref{thm:improved_rate_Hm1}, there holds for $0\leq l \leq k+1$
	\begin{align*}
			\|\Gaussh-\Kex\|_{\Hmone}
			&\lesssim h\big(\nrm{\gappr-\gex}_2+\inf\limits_{v_h\in\Vo_h^k}\|v_h-\Kex\|_{\Ltwo}+\|\gappr-\gex\|_{\Linf}\big)\\
			&\lesssim h^{l+1}\big(\|\gex\|_{W^{l,\infty}}+\|\Kex\|_{H^{l}}\big).
	\end{align*}
\end{theorem}
This yields also improved rates in stronger norms for the lifted Gauss curvature:
\begin{corollary}
\label{cor:improved_conv_curvature}
Under the assumptions of Theorem~\ref{thm:improved_rate_Hm1}, there holds for all $0\leq l\leq k+1$ and $0\leq r\leq l$
\begin{align*}
		\|\Gaussh-\Kex\|_{H^r_h}&\lesssim h^{-r}\big(\|\gappr-\gex\|_{\Linf[]}+\nrm{\gappr-\gex}_2+\inf\limits_{v_h\in\Vo_h^{k}}\|v_h-\Kex\|_{\Ltwo}+h^{l}|\Kex|_{H^{l}}\big)\\
	&\lesssim h^{l-r}\big(\|\gex\|_{W^{l,\infty}}+\|\Kex\|_{H^{l}}\big).
\end{align*}
\end{corollary}
\begin{proof}
	Follows analogously to the proof of \cite[Corrolary 6.6]{GNSW2023} and \cite[p. 1818]{Gaw20} as the error in stronger Sobolev norms is traced back to the $\Hmone$-norm. See also the proof of Corollary~\ref{cor:improved_conv_dens_curvature} below.
\end{proof}

\subsection{Proof of Corollary~\ref{cor:improved_conv_dens_curvature}}
To prove the desired rates for the densitized Gauss curvature we first note that for the $\Ltwo$-norm there directly holds with Lemma~\ref{lem:approx_diff_volforms}, $\|\Gaussh\|_{H^1_h}\lesssim 1$, and Corollary~\ref{cor:improved_conv_curvature} for $0\le l\le k+1$
\begin{align*}
		\|\Gaussh\volformh-\Kex\volformex \|_{\Ltwo} &\lesssim \|\Gaussh(\volformh-\volformex)\|_{\Ltwo}+\|\Gaussh-\Kex\|_{\Ltwo}\\
		&\lesssim h\|\gappr-\gex\|_{\Linf}+\|\gappr-\gex\|_{\Linf[]}+\nrm{\gappr-\gex}_2+\inf\limits_{v_h\in\Vo_h^{k}}\|v_h-\Kex\|_{\Ltwo}+h^{l}|\Kex|_{H^l}\\
		&\lesssim h^{l}\big(\|\gex\|_{W^{l,\infty}}+\|\Kex\|_{H^{l}}\big).
\end{align*}

With the improved $\Ltwo$ error estimate at hand we can prove optimal convergence rates in stronger Sobolev spaces.
\begin{proof}[Proof of Corollary~\ref{cor:improved_conv_dens_curvature}]
	Let $u_h\in \Vo_h^{k}$ be the Scott--Zhang interpolant \cite{SZ1990} of $\Kex\volformex$. Then there holds, analogously to the proof of \cite[p. 1818]{Gaw20},
	\begin{align*}
		|\Gaussh\volformh-\Kex\volformex|_{H^r_h}&\leq |\Gaussh\volformh-u_h|_{H^r_h} + |u_h-\Kex\volformex|_{H^r_h}\\
		&\lesssim h^{-r}\|\Gaussh\volformh-u_h\|_{\Ltwo} + h^{l-r}\|\Kex\volformex\|_{H^{l}}\\
		&\leq h^{-r}\big(\|\Gaussh\volformh-\Kex\volformex\|_{\Ltwo}+\|\Kex\volformex-u_h\|_{\Ltwo}+ h^{l}\|\Kex\volformex\|_{H^{l}}\big)\\
		&\lesssim h^{-r}\big(\|\gappr-\gex\|_{\Linf[]}+\nrm{\gappr-\gex}_2+\inf\limits_{v_h\in\Vo_h^{k}}\|v_h-\Kex\|_{\Ltwo}\\
		&\qquad\qquad+h^{l}\|\Kex\|_{H^{l}}+\inf\limits_{v_h\in\Vo_h^{k}}\|v_h-\Kex\volformex\|_{\Ltwo}+h^{l}\|\Kex\volformex\|_{H^{l}}\big)\\
		&\lesssim h^{l-r}\big(\|\gex\|_{W^{l,\infty}}+ \|\Kex\|_{H^{l}} + \|\Kex\volformex\|_{H^{l}}\big).
	\end{align*}
\end{proof}

\section{Numerical examples}
\label{sec:num_examples}

In this section we confirm, by numerical examples, that the theoretical convergence rates from Theorem~\ref{thm:improved_rate_Hm1}, Corollary~\ref{cor:improved_conv_dens_curvature}, Theorem~\ref{thm:improved_rate_Hm1_pure}, and Corollary~\ref{cor:improved_conv_curvature} are sharp.
All experiments were performed in the open source finite element software NGSolve\footnote{\href{www.ngsolve.org}{www.ngsolve.org}} \cite{Sch97,Sch14}, where the Regge elements are available.

We consider the numerical example proposed in \cite{Gaw20}, where on the square $\Omega=(-1,1)\times (-1,1)$ the smooth Riemannian metric tensor
\begin{align*}
	\gex(x,y)= \mat{1+(\pder{f}{x})^2 & \pder{f}{x}\pder{f}{y} \\ \pder{f}{x}\pder{f}{y} & 1+ (\pder{f}{y})^2}
\end{align*}
with $f(x,y)= \frac{1}{2}(x^2+y^2)-\frac{1}{12}(x^4+y^4)$ is defined. This metric corresponds to the surface induced by the embedding $\big(x,y\big)\mapsto \big(x,y,f(x,y)\big)$ and its exact Gauss curvature is given by
\begin{align*}
	\Kex(x,y) = \frac{81(1-x^2)(1-y^2)}{(9+x^2(x^2-3)^2+y^2(y^2-3)^2)^2}.
\end{align*}

To test also the case of non-homogeneous Dirichlet and Neumann boundary conditions we follow \cite{GNSW2023} and consider only one quarter $\Omega=(0,1)\times(0,1)$ and define the right and bottom boundaries as Dirichlet and the remaining parts as Neumann boundary.
We start with a structured mesh consisting of $2^{2l+1}$ triangles with maximal mesh-size $h=\max_T h_T=\sqrt{2}\,2^{-l}$ (and minimal edge length $2^{-l}$) for $l=0,1,\dots$. To avoid possible super-convergence properties due to a structured grid, we perturb all internal points of the triangular mesh by a uniform distribution in the range $[-\frac{h}{2^{2.5}},\frac{h}{2^{2.5}}]$. The geodesic curvature on the left boundary is exactly zero, whereas on the top boundary we have
\begin{align*}
	\dhat{\GeodCurv}|_{\Gamma_{\mathrm{left}}} = 0,\qquad\dhat{\GeodCurv}|_{\Gamma_{\mathrm{top}}} = \frac{-27(x^2 - 1)y(y^2 - 3)}{(x^2(x^2 - 3)^2 + 9)^{3/2}\sqrt{x^2(x^2 - 3)^2 + y^2(y^2 - 3)^2 + 9}}.
\end{align*}
The vertex expressions $\dhat{\sphericalangle}_V^N$ at the vertices of the Neumann boundary can directly be computed by measuring the angle $\arccos( \gex(\gt_V^1, \gt_V^2) )$.

To illustrate our theorems, we must use
$\gappr=\RegInt[k]\gex$.  In implementing the Regge interpolant, the
moments on the edges have to coincide exactly: see \eqref{eq:RegInt}.
To this end,  we use a high enough integration rule for interpolating $\gex$ for  minimizing the numerical integration errors.

We compute and report the curvature error in the  $\Ltwo$-norm, namely  $\|\Kex-\Gaussh\|_{\Ltwo[]}$ and $\|\Kex\volformex-\Gaussh\volformh\|_{\Ltwo[]}$. We also report the  $\Hmone$-norm of the errors. They can be computed by solving e.g. for  $w\in\Honez[\Omega]$ such that 
$-\Delta w= \Kex-\Gaussh$ and observing that 
\begin{align*}
	\|\Kex-\Gaussh\|_{\Hmone[]} = \|w\|_{\Hone[]}.
\end{align*}
Of course the right-hand side can generally be computed only
approximately.  To avoid extraneous errors, we approximate $w$ using
Lagrange finite elements of two degrees more, i.e.,
$w_h\in \VV_h^{k+2}$ when  $\Gauss_h\in \VV_h^{k}$.

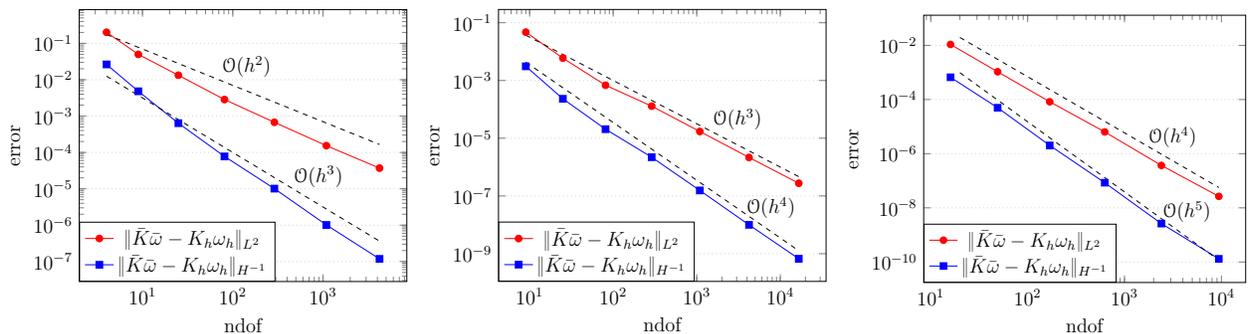
\begin{figure}[!htbp]
	\centering
	\resizebox{0.32\linewidth}{!}{
		\begin{tikzpicture}
			\begin{loglogaxis}[
				legend style={at={(0,0)}, anchor=south west},
				xlabel={ndof},
				ylabel={error},
				ymajorgrids=true,
				grid style=dotted,
				]
				\addlegendentry{$\|\Kex\volformex-\Gaussh\volformh\|_{L^2}$}
				\addplot[color=red, mark=*] coordinates {
					( 4 , 0.20198870270538344 )
					( 9 , 0.04977997751334723 )
					( 25 , 0.013306687975696233 )
					( 81 , 0.0028526823906146726 )
					( 289 , 0.0006760255845594408 )
					( 1089 , 0.0001536597017701977 )
					( 4225 , 3.707447978638317e-05 )
				};
				\addlegendentry{$\|\Kex\volformex-\Gaussh\volformh\|_{H^{-1}}$}
				\addplot[color=blue, mark=square*] coordinates {
					( 4 , 0.026384458097065036 )
					( 9 , 0.004808882524132786 )
					( 25 , 0.0006384651260557073 )
					( 81 , 7.775183865671265e-05 )
					( 289 , 1.012912536754616e-05 )
					( 1089 , 1.014672085767427e-06 )
					( 4225 , 1.1837048340964129e-07 )
				};

				\def\scal{0.7}
				\addplot[dashed, color=black] coordinates {
					( 4 , \scal*0.25 )
					( 4225 , \scal*0.00023668639053254438 )
				};
				
				\def\scal{0.1}
				\addplot[dashed, color=black] coordinates {
					( 4 , \scal*0.125 )
					( 4225 , \scal*3.6413290851160674e-06 )
				};
				
			\end{loglogaxis}
			\node (A) at (3.5, 4.5) [] {$\mathcal{O}(h^2)$};
			\node (B) at (5, 2.2) [] {$\mathcal{O}(h^3)$};
	\end{tikzpicture}}
	\resizebox{0.32\linewidth}{!}{
		\begin{tikzpicture}
			\begin{loglogaxis}[
				legend style={at={(0,0)}, anchor=south west},
				xlabel={ndof},
				ylabel={error},
				ymajorgrids=true,
				grid style=dotted,
				]
				\addlegendentry{$\|\Kex\volformex-\Gaussh\volformh\|_{L^2}$}
				\addplot[color=red, mark=*] coordinates {
					( 9 , 0.04722020277962978 )
					( 25 , 0.005956532768753863 )
					( 81 , 0.0006843689750579143 )
					( 289 , 0.00012902780878191577 )
					( 1089 , 1.7126847362316526e-05 )
					( 4225 , 2.1549317894270887e-06 )
					( 16641 , 2.756122089034252e-07 )
				};
				\addlegendentry{$\|\Kex\volformex-\Gaussh\volformh\|_{H^{-1}}$}
				\addplot[color=blue, mark=square*] coordinates {
					( 9 , 0.0030949689764758663 )
					( 25 , 0.00023225190833558144 )
					( 81 , 2.0550393928576182e-05 )
					( 289 , 2.196605571049263e-06 )
					( 1089 , 1.553434206326641e-07 )
					( 4225 , 9.981131402102656e-09 )
					( 16641 , 6.689738708610919e-10 )
				};

				\def\scal{1}
				\addplot[dashed, color=black] coordinates {
					( 9 , \scal*0.037037037037037035 )
					( 16641 , \scal*4.6583366291064987e-07 )
				};
				
				\def\scal{0.35}
				\addplot[dashed, color=black] coordinates {
					( 9 , \scal*0.012345679012345678 )
					( 16641 , \scal*3.6111136659740296e-09 )
				};
				
			\end{loglogaxis}
			\node (A) at (5, 3.4) [] {$\mathcal{O}(h^3)$};
			\node (A) at (5.7, 1.55) [] {$\mathcal{O}(h^4)$};
	\end{tikzpicture}}
	\resizebox{0.32\linewidth}{!}{
		\begin{tikzpicture}
			\begin{loglogaxis}[
				legend style={at={(0,0)}, anchor=south west},
				xlabel={ndof},
				ylabel={error},
				ymajorgrids=true,
				grid style=dotted,
				]
				\addlegendentry{$\|\Kex\volformex-\Gaussh\volformh\|_{L^2}$}
				\addplot[color=red, mark=*] coordinates {
					( 16 , 0.010821563111129457 )
					( 49 , 0.0010616329023298374 )
					( 169 , 8.286067944941594e-05 )
					( 625 , 6.380074488365906e-06 )
					( 2401 , 3.688038608568743e-07 )
					( 9409 , 2.6575250665613493e-08 )
				};
				\addlegendentry{$\|\Kex\volformex-\Gaussh\volformh\|_{H^{-1}}$}
				\addplot[color=blue, mark=square*] coordinates {
					( 16 , 0.0006633044741982852 )
					( 49 , 4.9936505896321055e-05 )
					( 169 , 2.0215639040713894e-06 )
					( 625 , 8.483689913933702e-08 )
					( 2401 , 2.623851025279834e-09 )
					( 9409 , 1.3130625631342235e-10 )
				};

				\def\scal{5}
				\addplot[dashed, color=black] coordinates {
					( 20 , \scal*0.00390625 )
					( 9409 , \scal*1.1295697747731624e-08 )
				};
				
				\def\scal{1}
				\addplot[dashed, color=black] coordinates {
					( 20 , \scal*0.0009765625 )
					( 9409 , \scal*1.164504922446559e-10 )
				};
			\end{loglogaxis}
			\node (A) at (5.3, 3.1) [] {$\mathcal{O}(h^4)$};
			\node (A) at (5.6, 1.5) [] {$\mathcal{O}(h^5)$};
	\end{tikzpicture}}
	
	\caption{Convergence of lifted densitized Gauss curvature with respect to number of degrees of freedom (ndof) in different norms for Regge elements $\gappr\in\Regge_h^k$ of order $k=1,2,3$.}
	\label{fig:conv_plot_dens_curvature}
\end{figure}

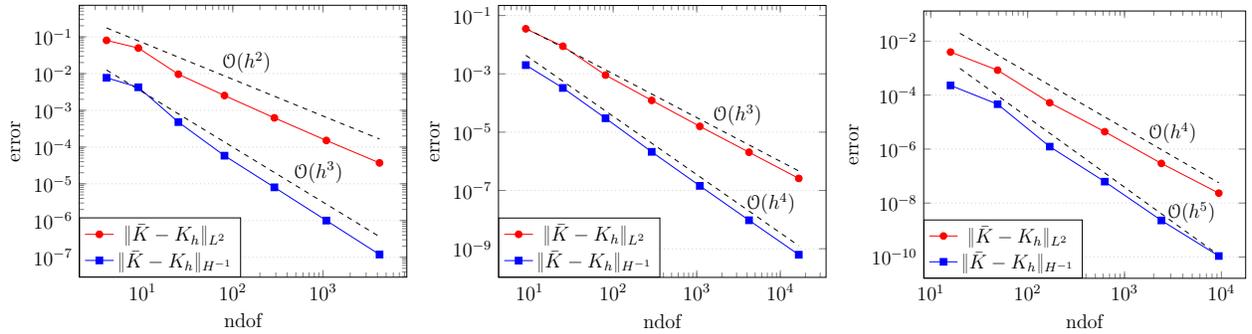
\begin{figure}[!htbp]
	\centering
	\resizebox{0.32\linewidth}{!}{
		\begin{tikzpicture}
			\begin{loglogaxis}[
				legend style={at={(0,0)}, anchor=south west},
				xlabel={ndof},
				ylabel={error},
				ymajorgrids=true,
				grid style=dotted,
				]
				\addlegendentry{$\|\Kex-\Gaussh\|_{L^2}$}
				\addplot[color=red, mark=*] coordinates {
					( 4 , 0.08020568758129472 )
					( 9 , 0.049675897242784116 )
					( 25 , 0.009601106770691379 )
					( 81 , 0.0025274171385003667 )
					( 289 , 0.0006243043743432753 )
					( 1089 , 0.00015103722826237745 )
					( 4225 , 3.707811340303871e-05 )
				};
				\addlegendentry{$\|\Kex-\Gaussh\|_{H^{-1}}$}
				\addplot[color=blue, mark=square*] coordinates {
					( 4 , 0.007702629636405941 )
					( 9 , 0.004239102948909943 )
					( 25 , 0.0004784254037894367 )
					( 81 , 5.819100803988829e-05 )
					( 289 , 8.006173608579616e-06 )
					( 1089 , 9.979945983178577e-07 )
					( 4225 , 1.1865326680477532e-07 )
				};

				\def\scal{0.7}
				\addplot[dashed, color=black] coordinates {
					( 4 , \scal*0.25 )
					( 4225 , \scal*0.00023668639053254438 )
				};
			
				\def\scal{0.1}
				\addplot[dashed, color=black] coordinates {
					( 4 , \scal*0.125 )
					( 4225 , \scal*3.6413290851160674e-06 )
				};
				
			\end{loglogaxis}
			\node (A) at (3.5, 4.5) [] {$\mathcal{O}(h^2)$};
			\node (B) at (5, 2.2) [] {$\mathcal{O}(h^3)$};
	\end{tikzpicture}}
	\resizebox{0.32\linewidth}{!}{
		\begin{tikzpicture}
			\begin{loglogaxis}[
				legend style={at={(0,0)}, anchor=south west},
				xlabel={ndof},
				ylabel={error},
				ymajorgrids=true,
				grid style=dotted,
				]
				\addlegendentry{$\|\Kex-\Gaussh\|_{L^2}$}
				\addplot[color=red, mark=*] coordinates {
					( 9 , 0.0350453867551159 )
					( 25 , 0.008842988372418061 )
					( 81 , 0.0009055260151918672 )
					( 289 , 0.00012281884773299965 )
					( 1089 , 1.5762402968190808e-05 )
					( 4225 , 2.044276052601903e-06 )
					( 16641 , 2.587697120268734e-07 )
				};
				\addlegendentry{$\|\Kex-\Gaussh\|_{H^{-1}}$}
				\addplot[color=blue, mark=square*] coordinates {

					( 9 , 0.0019943950498126297 )
					( 25 , 0.0003280070724256844 )
					( 81 , 2.9973438339446926e-05 )
					( 289 , 2.109293368535958e-06 )
					( 1089 , 1.435925157394553e-07 )
					( 4225 , 9.483241913642543e-09 )
					( 16641 , 6.24812635164102e-10 )
				};

				\def\scal{1}
				\addplot[dashed, color=black] coordinates {
					( 9 , \scal*0.037037037037037035 )
					( 16641 , \scal*4.6583366291064987e-07 )
				};
				
				\def\scal{0.35}
				\addplot[dashed, color=black] coordinates {
					( 9 , \scal*0.012345679012345678 )
					( 16641 , \scal*3.6111136659740296e-09 )
				};
				
			\end{loglogaxis}
			\node (A) at (5, 3.4) [] {$\mathcal{O}(h^3)$};
			\node (A) at (5.7, 1.55) [] {$\mathcal{O}(h^4)$};
	\end{tikzpicture}}
	\resizebox{0.32\linewidth}{!}{
		\begin{tikzpicture}
			\begin{loglogaxis}[
				legend style={at={(0,0)}, anchor=south west},
				xlabel={ndof},
				ylabel={error},
				ymajorgrids=true,
				grid style=dotted,
				]
				\addlegendentry{$\|\Kex-\Gaussh\|_{L^2}$}
				\addplot[color=red, mark=*] coordinates {
					( 16 , 0.003950950394616541 )
					( 49 , 0.0008375113652966704 )
					( 169 , 5.218871554774373e-05 )
					( 625 , 4.404615769679878e-06 )
					( 2401 , 2.942620281225721e-07 )
					( 9409 , 2.3350971439171375e-08 )
				};
				\addlegendentry{$\|\Kex-\Gaussh\|_{H^{-1}}$}
				\addplot[color=blue, mark=square*] coordinates {
					( 16 , 0.00022845229558764236 )
					( 49 , 4.602078699200278e-05 )
					( 169 , 1.2348436288235488e-06 )
					( 625 , 6.210583258899291e-08 )
					( 2401 , 2.2660870707671694e-09 )
					( 9409 , 1.0909461187383017e-10 )
				};

				\def\scal{5}
				\addplot[dashed, color=black] coordinates {
					( 20 , \scal*0.00390625 )
					( 9409 , \scal*1.1295697747731624e-08 )
				};
				
				\def\scal{1}
				\addplot[dashed, color=black] coordinates {
					( 20 , \scal*0.0009765625 )
					( 9409 , \scal*1.164504922446559e-10 )
				};
			\end{loglogaxis}
			\node (A) at (5.3, 3.1) [] {$\mathcal{O}(h^4)$};
			\node (A) at (5.7, 1.4) [] {$\mathcal{O}(h^5)$};
	\end{tikzpicture}}
	
	\caption{Convergence of pure lifted Gauss curvature with respect to number of degrees of freedom (ndof) in different norms for Regge elements $\gappr\in\Regge_h^k$ of order $k=1,2,3$.}
	\label{fig:conv_plot_curvature}
\end{figure}
We start by approximating $\gex$ with linear Regge elements $\gappr\in\Regge_h^1$. As shown in Figure~\ref{fig:conv_plot_dens_curvature} (left), we obtain the stated quadratic convergence in the $\Ltwo$-norm and cubic rate in the weaker $\Hmone$-norm, in agreement with Theorem~\ref{thm:improved_rate_Hm1}. When increasing the approximation order of Regge elements to quadratic and cubic polynomials we observe the appropriate increase of the convergence rates: see  Figure~\ref{fig:conv_plot_dens_curvature} (middle and right), confirming that the results stated in Theorem~\ref{thm:improved_rate_Hm1} and Corollary~\ref{cor:improved_conv_dens_curvature} are sharp. For the error of the pure Gauss curvature we practically obtain the same behavior as stated by Theorem~\ref{thm:improved_rate_Hm1_pure} and Corollary~\ref{cor:improved_conv_curvature}, cf. Figure~\ref{fig:conv_plot_curvature}. Only in the pre-asymptotic regime the error is smaller compared to the densitized Gauss curvature.

We conclude with a few additional remarks on the lifting degree.	Attempting to  increase the degree for the curvature approximation, say by placing  $\Gaussh$ in $\VV_h^{k+1}$ or $\VV_h^{k+2}$, while the metric $\gappr$ remains in $ \Regge_h^k$, need not generally produce additional orders of convergence.
	This is because the orthogonality properties of the canonical Regge interpolant, namely~\eqref{eq:RefInt_edge}--\eqref{eq:RefInt_trig}, may not be fulfilled in such cases. Indeed, we numerically observed loss of two orders of convergence when $\Gaussh$ is placed in $\VV_h^{k+2}$ instead of $\VV_h^k$. In \cite{GNSW2023}, where we used  $\Gaussh\in \VV_h^{k+1},$ one order less is obtained, again due to the orthogonality properties of the canonical Regge interpolant.

	Finally, when reducing  the polynomial degree of the curvature approximation from $\VV_h^k$ to $\VV_h^{k-1}$, $k\ge 2$, while keeping the metric in $\Regge_h^k$, we observed that the convergence rates reduce by one order. 
        Note that the orthogonality properties of the canonical Regge interpolant are still fulfilled. Nevertheless, the overall approximation ability of the space  is reduced so that the convergence rate in the $\Hmone$-norm decreases from $\mathcal{O}(h^{k+2})$ to $\mathcal{O}(h^{k+1})$.

\section*{Acknowledgments}
This work was supported in part by the Austrian Science Fund (FWF) project \href{https://doi.org/10.55776/F65}{10.55776/F65}
and the National Science Foundation (USA) Grant DMS-2409900.

\bibliographystyle{acm}
\bibliography{cites}

\begin{thebibliography}{10}

\bibitem{Arnol18}
{\sc Arnold, D.~N.}
\newblock {\em Finite element exterior calculus}.
\newblock SIAM, Philadelphia, 2018.

\bibitem{BKG21}
{\sc Berchenko-Kogan, Y., and Gawlik, E.~S.}
\newblock Finite element approximation of the {L}evi-{C}ivita connection and its curvature in two dimensions.
\newblock {\em Foundations of Computational Mathematics\/} (2022).

\bibitem{BS2008}
{\sc Brenner, S.~C., and Scott, L.~R.}
\newblock {\em The {{Mathematical Theory}} of {{Finite Element Methods}}}, vol.~15 of {\em Texts in {{Applied Mathematics}}}.
\newblock {Springer New York}, {New York, NY}, 2008.

\bibitem{Christiansen04}
{\sc Christiansen, S.~H.}
\newblock A characterization of second-order differential operators on finite element spaces.
\newblock {\em Mathematical Models and Methods in Applied Sciences 14}, 12 (2004), 1881--1892.

\bibitem{Christiansen11}
{\sc Christiansen, S.~H.}
\newblock On the linearization of {R}egge calculus.
\newblock {\em Numerische Mathematik 119}, 4 (2011), 613--640.

\bibitem{Christiansen13}
{\sc Christiansen, S.~H.}
\newblock Exact formulas for the approximation of connections and curvature.
\newblock {\em arXiv preprint arXiv:1307.3376\/} (2013).

\bibitem{Chr2024}
{\sc Christiansen, S.~H.}
\newblock On the definition of curvature in {{Regge}} calculus.
\newblock {\em IMA Journal of Numerical Analysis\/} (2024), drad095.

\bibitem{CrouzThome87}
{\sc Crouzeix, M., and Thom{\'{e}}e, V.}
\newblock {The Stability in $L_p$ and $W_p^1$ of the $L_2$-projection onto Finite Element Function Spaces}.
\newblock {\em Math. Comp 48}, 178 (1987), 521--532.

\bibitem{Gaw20}
{\sc Gawlik, E.~S.}
\newblock High-order approximation of {G}aussian curvature with {R}egge finite elements.
\newblock {\em SIAM Journal on Numerical Analysis 58}, 3 (2020), 1801--1821.

\bibitem{GN2023}
{\sc Gawlik, E.~S., and Neunteufel, M.}
\newblock Finite element approximation of scalar curvature in arbitrary dimension.
\newblock {\em arXiv preprint arXiv:2301.02159\/} (2023).

\bibitem{GN2023b}
{\sc Gawlik, E.~S., and Neunteufel, M.}
\newblock Finite element approximation of the {E}instein tensor.
\newblock {\em arXiv preprint arXiv:2310.18802\/} (2023).

\bibitem{GNSW2023}
{\sc Gopalakrishnan, J., Neunteufel, M., Sch{\"o}berl, J., and Wardetzky, M.}
\newblock Analysis of curvature approximations via covariant curl and incompatibility for {{Regge}} metrics.
\newblock {\em The SMAI Journal of computational mathematics 9\/} (2023), 151--195.

\bibitem{GNSW2023b}
{\sc Gopalakrishnan, J., Neunteufel, M., Sch{\"o}berl, J., and Wardetzky, M.}
\newblock Analysis of distributional {R}iemann curvature tensor in any dimension.
\newblock {\em arXiv preprint arXiv:2311.01603\/} (2023).

\bibitem{Hel67}
{\sc Hellan, K.}
\newblock Analysis of elastic plates in flexure by a simplified finite element method.
\newblock {\em Acta Polytechnica Scandinavica, Civil Engineering Series 46\/} (1967).

\bibitem{Her67}
{\sc Herrmann, L.~R.}
\newblock Finite element bending analysis for plates.
\newblock {\em Journal of the Engineering Mechanics Division 93}, 5 (1967), 13--26.

\bibitem{Joh73}
{\sc Johnson, C.}
\newblock On the convergence of a mixed finite element method for plate bending moments.
\newblock {\em Numerische Mathematik 21}, 1 (1973), 43--62.

\bibitem{Lee97}
{\sc Lee, J.~M.}
\newblock {\em Riemannian manifolds: an introduction to curvature}, 1~ed.
\newblock Springer, New York, NY, New York, 1997.

\bibitem{li18}
{\sc Li, L.}
\newblock {\em Regge Finite Elements with Applications in Solid Mechanics and Relativity}.
\newblock PhD thesis, University of Minnesota, 2018.

\bibitem{Neun21}
{\sc Neunteufel, M.}
\newblock {\em Mixed Finite Element Methods for Nonlinear Continuum Mechanics and Shells}.
\newblock PhD thesis, TU Wien, 2021.

\bibitem{Peter16}
{\sc Petersen, P.}
\newblock {\em Riemannian Geometry}, third~ed.
\newblock Springer, 2016.

\bibitem{Regge61}
{\sc Regge, T.}
\newblock General relativity without coordinates.
\newblock {\em Il Nuovo Cimento (1955-1965) 19}, 3 (1961), 558--571.

\bibitem{Sch97}
{\sc Sch{\"o}berl, J.}
\newblock {NETGEN} an advancing front 2{D}/3{D}-mesh generator based on abstract rules.
\newblock {\em Computing and Visualization in Science 1}, 1 (1997), 41--52.

\bibitem{Sch14}
{\sc Sch{\"o}berl, J.}
\newblock C++ 11 implementation of finite elements in {NGS}olve.
\newblock {\em Institute for Analysis and Scientific Computing, Vienna University of Technology\/} (2014).

\bibitem{SZ1990}
{\sc Scott, L.~R., and Zhang, S.}
\newblock Finite element interpolation of nonsmooth functions satisfying boundary conditions.
\newblock {\em Mathematics of Computation 54}, 190 (1990), 483--493.

\bibitem{Sorkin75}
{\sc Sorkin, R.}
\newblock Time-evolution problem in {R}egge calculus.
\newblock {\em Phys. Rev. D 12\/} (1975), 385--396.

\bibitem{Str2020}
{\sc Strichartz, R.~S.}
\newblock Defining {{Curvature}} as a {{Measure}} via {{Gauss}}\textendash{{Bonnet}} on {{Certain Singular Surfaces}}.
\newblock {\em The Journal of Geometric Analysis 30}, 1 (2020), 153--160.

\bibitem{Sullivan08}
{\sc Sullivan, J.~M.}
\newblock {\em Curvatures of Smooth and Discrete Surfaces}.
\newblock Birkh{\"a}user Basel, Basel, 2008, pp.~175--188.

\bibitem{Tu17}
{\sc Tu, L.~W.}
\newblock {\em Differential Geometry: Connections, Curvature, Characteristic Classes}.
\newblock Springer, 2017.

\bibitem{whitney57}
{\sc Whitney, H.}
\newblock {\em Geometric integration theory}.
\newblock Princeton University Press, Princeton, N. J, 1957.

\end{thebibliography}

\end{document}